\documentclass{amsart}
\usepackage{amsmath}
\usepackage{amssymb}
\usepackage{latexsym}
\usepackage{amsbsy}
\usepackage[all]{xy}
\usepackage{amscd}
\usepackage[dvips]{graphicx}
\usepackage{color}

\newtheorem{theorem}{Theorem}[section]
\newtheorem{lemma}[theorem]{Lemma}

\theoremstyle{definition}

\newtheorem{Prop}[theorem]{Proposition}
\newtheorem{Cor}[theorem]{Corollary}

\numberwithin{equation}{section}

\newcommand\bbc{{\mathbb C}}

\newcommand\md{{\mathcal{D}}}

\newcommand\mo{{\mathcal{O}}}

\begin{document}

\title[Ringel-Hall algebras beyond their quantum groups]{Ringel-Hall algebras beyond their quantum groups I: Restriction functor and Green formula}

\author[Xiao, Xu, Zhao]{Jie Xiao, Fan Xu, Minghui Zhao\\
\\
\\
\emph{\textbf{To the memory of Professor J. A. Green.}}}

\address{Department of Mathematics\\ Tsinghua University,
Beijing {\rm 100084}, P. R. China} \email{jxiao@math.tsinghua.edu.cn}

\address{Department of Mathematics\\ Tsinghua University,
Beijing {\rm 100084}, P. R. China} \email{fanxu@mail.tsinghua.edu.cn}

\address{College of Science\\ Beijing Forestry University,
Beijing {\rm 100083}, P. R. China} \email{zhaomh@bjfu.edu.cn}

\thanks{Jie Xiao was supported by NSF of China (No. 11131001), Fan Xu was supported by NSF of China (No. 11471177) and Minghui Zhao was supported by NSF of China (No. 11701028, 11771445).}

\subjclass[2010]{16G20 (primary), 17B37 (secondary).}

\date{\today}

\keywords{Hall algebra, Green formula, induction functor, restriction functor, simple perverse sheaf.}

\begin{abstract}
In this paper, we generalize the categorifical construction of a quantum group and its canonical basis introduced by Lusztig to the generic form of the whole Ringel-Hall algebra. We clarify the explicit relation between the Green formula and the restriction functor. By a geometric way to prove the Green formula, we show that  the compatibility of multiplication and comultiplication  of a Ringel-Hall algebra can be categorified under Lusztig's framework.
\end{abstract}

\maketitle

\tableofcontents

\section{Introduction} 
\label{intro}
The Ringel-Hall algebra $\mathcal{H}(\mathcal{A})$ of a (small) abelian category $\mathcal{A}$ was introduced by Ringel in \cite{Ringel1990}, as a model to realize the quantum group. When $\mathcal{A}$ is the category $\mathrm{Rep}_{\mathbb{F}_q}Q$ of finite dimensional representations for a simply-laced Dynkin quiver $Q$ over a finite field $\mathbb{F}_q$, the Ringel-Hall algebra $\mathcal{H}(\mathcal{A})$ is isomorphic to the positive/negative part of the corresponding quantum group (\cite{Ringel1990}).  For any acyclic quiver $Q$ and $\mathcal{A}=\mathrm{Rep}_{\mathbb{F}_q}Q$, the composition subalgebra of $\mathcal{H}(\mathcal{A})$ generated by the elements corresponding to simple representations is isomorphic to the positive/negative part of the  quantum group of type $Q$. This gives the algebraic realization of the positive/negative part of a (Kac-Moody type) quantum group. This realization was improved by Green (\cite{Green}), through solving a natural question whether there is a comultiplication on $\mathcal{H}(\mathcal{A})$ compatible with the corresponding multiplication so that the above isomorphism is an isomorphism between bialgebras. Now it is well-known that Green's comultiplication depends on a remarkable homological formula in \cite{Green} (called the Green formula in the following).

In the earlier seminal papers \cite{Lusztig} and \cite{Lusztig2}, Lusztig gave the geometric realization of the positive/negative part of a quantum group and then constructed the canonical basis for it.
Let $Q=(Q_0,Q_1, s, t)$ be a quiver and
$$
\mathbb{E}_{\alpha}:=\bigoplus_{h\in Q_1}\mathrm{Hom}_{\mathbb{K}}(\mathbb{K}^{\alpha_{s(h)}}, \mathbb{K}^{\alpha_{t(h)}})
$$ be the variety with the natural action of the algebraic group
$$G_{\alpha}:=\prod_{i\in Q_0}GL(\alpha_i,\mathbb{K})$$
for a given dimension vector $\alpha=\sum_{i\in Q_0}\alpha_ii\in \mathbb{N}Q_0.$
For any $\mathbf{i}=(i_1,i_2,\cdots,i_s)$, $i_l\in Q_0$ and $\mathbf{a}=(a_1,a_2,\cdots,a_s)$, $a_l\in \mathbb{N}$ such that $\sum_{l=1}^{s}a_li_l=\alpha$,
Lusztig (\cite{Lusztig2}) defined the flag variety ${F}_{\mathbf{i}, \mathbf{a}}$ and the subvariety $$\tilde{F}_{\mathbf{i}, \mathbf{a}}\subseteq\mathbb{E}_{\alpha}\times {F}_{\mathbf{i}, \mathbf{a}}.$$
Fixing any type $(\mathbf{i}, \mathbf{a}),$ consider the canonical proper morphism $\pi_{\mathbf{i}, \mathbf{a}}:\tilde{F}_{\mathbf{i}, \mathbf{a}}\rightarrow \mathbb{E}_{\alpha}.$ By the decomposition theorem of Beilinson, Bernstein and Deligne (\cite{BBD}), the complex $\pi_{\mathbf{i}, \mathbf{a}!}\mathbf{1}$ is semisimple, where $\mathbf{1}$ is the constant perverse sheaf on $\tilde{F}_{\mathbf{i}, \mathbf{a}}.$
Let $\mathcal{Q}_{\alpha}$ be the category of complexes isomorphic to sums of shifts of simple perverse sheaves appearing in $\pi_{\mathbf{i}, \mathbf{a}!}\mathbf{1},$
$K_{\alpha}$  the Grothendieck group of $\mathcal{Q}_{\alpha}$ and $$K(\mathcal{Q})=\bigoplus_{\alpha\in \mathbb{N}Q_0}K_{\alpha}.$$
Lusztig (\cite{Lusztig2}) already endowed $K(\mathcal{Q})$ with the multiplication and comultiplication structures by introducing his induction and restriction functors. He proved that the comultiplication is compatible with the multiplication in $K(\mathcal{Q})$ and  $K(\mathcal{Q})$ is isomorphic to the positive/negative part of  the corresponding quantum group as bialgebras up to a twist.

By this isomorphism, the isomorphism classes of simple perverse sheaves in $\mathcal{Q}_{\alpha}$ provide a basis of the positive/negative part of
the corresponding quantum group, which is called the canonical basis. The
canonical basis of a quantum group is crucially important in Lie theory.
This basis has many remarkable properties such as integrality and positivity
of structure constants, compatibility with all highest weight integrable representations,
etc. Lusztig's approach essentially motivates the categorification of quantum
groups (for example, see \cite{KL},\cite{Rouquier1} and \cite{VV}) or quantum cluster algebras (see \cite{HL},\cite{Nakajima2009},\cite{KQ}, etc.), i.e., a quantum group/quantum
cluster algebra can be viewed as the Grothendieck ring of a monoidal category and
some simple objects provide a basis (see also \cite{Webster}).

For a long time, we have been asked what the explicit relation exists between Green's comultiplication and Lusztig's restriction functor. As one of the main results in the present paper, the following Theorem \ref{thm_main} and the definition of the comultiplication operator $\Delta$ provide us this strong and clear link. Thanks to an embedding property as in \cite{KW}, we can lift the Green formula from finite fields to the level of sheaves. It is finally suitable to apply Lusztig's restriction functor to the larger categories of  $\iota$-mixed Weil complexes of integer weights, whose Grothendieck groups realize the weight spaces of a generic Ringel-Hall algebra. By using the direct sum of these Grothendieck groups, we also give the categorification of Ringel-Hall algebras via Lusztig's geometric method.

The paper is organized as follows. In Section 2, we recall the theory of Ringel-Hall algebras, focusing on the Hopf structure of a Ringel-Hall algebra.
In Section 3, we recall Lusztig's construction of a Hall algebra via functions invariant under the Frobenius map.
In \cite{Lusztig1998}, Lusztig defined the algebra $\mathcal{CF}^{F}(Q)$ with multiplication and comultiplication
by applying his restriction functor and induction functor to constructible functions.
However, the proof of the compatibility of Lusztig's comultiplication and  multiplication for the whole Ringel-Hall algebra essentially depends on the proof of the Green formula. In the end of this section, we show that the twist of $\mathcal{CF}^{F}(Q)$ is isomorphic to the twisted Ringel-Hall algebra $\mathcal{H}^{tw}(\mathcal{A})$. 
In Section 4, we extend the geometric realization of a quantum group  to the whole Ringel-Hall algebra under Lusztig's framework.  We obtain the generic Ringel-Hall algebra as the direct sum of Grothendieck groups of the derived categories of a class of Weil complexes. The simple perverse sheaves provide the canonical basis. We show that the compatibility of the induction and restriction functor holds for these perverse sheaves.
In Section 5, we come back to the case of quantum groups considered by Lusztig and Bozec. We point out that the algebras defined by them are subalgebras of the generic Ringel-Hall algebra and the canonical bases considered by them are subsets of our canonical basis.
We have defined the following map in Section 4,  $$\chi^{F^s}: \mathbf{K}_{w}\rightarrow \mathcal{CF}^{F^s}(Q).$$
In Section 6, we shall determinate the image of $\chi^{F^s}$ depending on Kac's and Hua's results.

\section{A revisit of Ringel-Hall algebras as Hopf algebras}
We recall the definition of the Ringel-Hall algebra $\mathcal{H}(\mathcal{A})$ for the hereditary abelian category $\mathcal{A}=\mathrm{mod} kQ=\mathrm{Rep}_{k}Q$, where $k=\mathbb{F}_q$ is a finite field with $q=p^e$ elements for some prime number $p$ and $Q$  is a finite quiver.

For $M\in \mathcal{A}$, we denote
by $\mathrm{\underline{dim}}M$ the dimension vector in $\mathbb{N}Q_0$ and define the Euler-Ringel form on $\mathbb{N}Q_0$ as follows:
$$
\langle \mathrm{\underline{dim}}M, \mathrm{\underline{dim}}N\rangle=\mathrm{dim}_k\mathrm{Hom}_{\mathcal{A}}(M, N)-\mathrm{dim}_k\mathrm{Ext}^1_{\mathcal{A}}(M, N).
$$
For $M, N$ and $L\in \mathrm{mod} kQ$, we denote by  $\mathcal{F}_{MN}^L$ the set $\{X\subset L\mid X\in \mathrm{mod} kQ, X\cong N, L/X\cong M\}$ and $\mathrm{Ext}_{\mathcal{A}}^1(M, N)_L$ the subset of $\mathrm{Ext}^1_{\mathcal{A}}(M, N)$ with the middle term isomorphic to $L$.  For $X\in \mathcal{A}$, denote by $\mathrm{Aut}_{\mathcal{A}}X$ the set of automorphism on $X$ in $\mathcal{A}$. Write $F_{MN}^L=|\mathcal{F}_{MN}^L|$, $h_{L}^{MN}=\frac{|\mathrm{Ext}_{\mathcal{A}}^1(M, N)_L|}{|\mathrm{Hom}_{\mathcal{A}}(M, N)|}$ and $a_X=|\mathrm{Aut}_{\mathcal{A}}X|$.

The ordinary Ringel-Hall algebra $\mathcal{H}(\mathcal{A})$ is a $\bbc$-space with isomorphism classes $[X]$ of all $kQ$-modules $X$ as a basis and the multiplication is defined by
$$
[M]*[N]=\sum_{[L]}F_{MN}^L[L]
$$
for $M, N$ and $L\in \mathrm{mod} kQ$. We can endow $\mathcal{H}=\mathcal{H}(\mathcal{A})$ with a comultiplication $\delta: \mathcal{H}\rightarrow  \mathcal{H}\otimes_{\bbc}\mathcal{H}$ by setting
$$
\delta([L])=\sum_{[M], [N]}h_{L}^{MN}[M]\otimes [N]
$$
for $L, M$ and $N\in \mathrm{mod} kQ$.
The comultiplication is compatible with the multiplication via Green's theorem.
\begin{theorem}\cite{Green}
The map $\delta$ is an algebra homomorphism with respect to the twisted multiplication on $\mathcal{H}\otimes \mathcal{H}$ as follows:
$$
([M_1]\otimes [N_1])\circ ([M_2]\otimes [N_2])=q^{-\langle\mathrm{\underline{dim}}M_1, \mathrm{\underline{dim}}N_2\rangle}([M_1]*[M_2])\otimes([N_1]*[N_2])
$$
for any $M_1, M_2, N_1$ and $N_2\in \mathrm{mod} kQ$.
\end{theorem}
The theorem is equivalent to  the following Green formula:
$$
a_{M_1}a_{M_2}a_{N_1}a_{N_2}\sum_{[L]}F_{M_1N_1}^LF_{M_2N_2}^La^{-1}_L$$$$=\sum_{[X],[Y_1],[Y_2],[Z]}\frac{|\mathrm{Ext}^1_{\mathcal{A}}(X,Z)|}{|\mathrm{Hom}_{\mathcal{A}}(X, Z)|}F_{XY_1}^{M_1}F_{XY_2}^{M_2}F_{Y_2Z}^{N_1}F_{Y_1Z}^{N_2}a_{X}a_{Y_1}a_{Y_2}a_{Z}.
$$

Define a symmetric bilinear form on $\mathcal{H}$ by setting $$([M], [N])=\frac{\delta_{M, N}}{a_M},$$
where $\delta_{M, N}$ is equal to $1$ if $M\cong N$ and $0$ otherwise. This form is called Green's Hopf pairing. It is clear that Green's Hopf pairing is a non degenerate bilinear form over $\mathcal{H}$.  The following proposition shows that the  comultiplication is dual to the multiplication, i.e., the comultiplication can be viewed as the multiplication over $\mathcal{H}^*=\mathrm{Hom}_{\bbc}(\mathcal{H}, \bbc).$
\begin{Prop}
The comultiplication is left adjoint to the multiplication with respect to Green's Hopf pairing, i.e.,
for $a, b, c\in \mathcal{H}$, $$(a, bc)=(\delta(a), b\otimes c ),$$ where
the bilinear form on $\mathcal{H}\otimes\mathcal{H}$ is given by $(a\otimes b,c\otimes d)=(a,c)(b,d)$ for any $a,b,c,d\in\mathcal{H}$.
\end{Prop}
The proposition is equivalent to that the Riedtmann-Peng formula
$$
F_{MN}^La_Ma_N=h_{L}^{MN}a_L
$$
holds for any $M, N$ and $L\in \mathrm{mod} kQ$.

It is easy to generalize the multiplication and comultiplication to the $r$-fold versions for $r\geq 2.$ For $M_1,\cdots, M_r, M\in \mathcal{A}$, set $\mathcal{F}_{M_1,\cdots,M_r}^M$ to be the set
$$
\{0=X_{0}\subseteq X_1\subseteq\cdots\subseteq X_r= M\mid X_i\in \mathcal{A}, X_{i+1}/X_i\cong M_{r-i}, i=0,1,\cdots, r-1\}
$$
and $F_{M_1\cdots M_r}^M=|\mathcal{F}_{M_1,\cdots,M_r}^M|.$ Then
$$
[M_1]*[M_2]*\cdots *[M_r]=\sum_{[M]}F_{M_1\cdots M_r}^M[M].
$$
The $r$-fold comultiplication $\delta^r$ can be defined inductively. For $r=1$, $\delta^1=\delta$ and $\delta^{r+1}=(1\otimes\cdots\otimes 1\otimes \delta)\circ \delta^r$ for $r\geq 1.$ Set
$$
\delta^{r-1}([M])=\sum_{[M_1],\cdots, [M_r]}h_{M}^{M_1\cdots M_r}[M_1]\otimes \cdots\otimes [M_r]
$$
for $r\geq 2.$ It is clear that the Riedtmann-Peng formula can be reformulated as
$$
h_{M}^{M_1M_2\cdots M_r}=F^M_{M_1\cdots M_r}a_{M_1}\cdots a_{M_r}a^{-1}_{M}
$$
for $r\geq 2.$

Let $\sigma: \mathcal{H}\rightarrow \mathcal{H}$ be a map such that
$$
\sigma([M])=\delta_{M,0}+\sum_{r\geq 1}(-1)^r\cdot \sum_{[N],[M_1], \cdots, [M_r]\neq 0}h_{M}^{M_1\cdots M_r}F_{M_1\cdots M_r}^{N}[N]
$$
for $M\in \mathcal{A}$. We call $\sigma$ the antipode of $\mathcal{H}.$

One can also define the twisted versions of the multiplication and comultiplication over $\mathcal{H}(\mathcal{A})$ by setting
$$
[M]\cdot[N]=v_q^{\langle \mathrm{\underline{dim}}M, \mathrm{\underline{dim}}N\rangle}[M]*[N]
$$
and
$$
\delta^t([L])=\sum_{[M], [N]}v_q^{\langle \mathrm{\underline{dim}}M, \mathrm{\underline{dim}}N\rangle}h_{L}^{MN}[M]\otimes [N],
$$
where $v_q=\sqrt{q}$.

Similarly, the $r$-fold twisted version of the multiplication is denoted by $[M_1]\cdot[M_2]\cdots[M_r]$ and the $r$-fold twisted version of the comultiplication is denoted by $\delta^{t,r}$.
The twisted version of the antipode over $\mathcal{H}(\mathcal{A})$ is defined as
$$\sigma^t([M])=\delta_{M,0}+\sum_{r\geq 1}(-1)^r\cdot \sum_{[N],[M_1], \cdots, [M_r]\neq 0}v_q^{2\sum_{i\leq j}\langle \mathrm{\underline{dim}}M_i, \mathrm{\underline{dim}}M_j\rangle}h_{M}^{M_1\cdots M_r}F_{M_1\cdots M_r}^{N}[N].$$
We denote by $\mathcal{H}^{tw}(\mathcal{A})$ the twisted version of $\mathcal{H}(\mathcal{A}).$

On the relations between the twisted version of antipode with the twisted versions of multiplication and comultiplication, we have
$$\sigma^t(x\cdot y)=\sigma^t(y)\cdot\sigma^t(x) \textrm{ for any $x,y\in\mathcal{H}(\mathcal{A})$},$$
$$\delta^t(\sigma^t(x))=(\sigma^t\otimes\sigma^t){\delta^{t,op}}(x) \textrm{ for any $x\in\mathcal{H}(\mathcal{A})$},$$
and
$$
\underline{m}^t(\sigma^t\otimes 1)\delta^t([M])=\underline{m}^t(1\otimes\sigma^t)\delta^t([M])
=\left\{
\begin{array}{c}
\textrm{$0$ \,\,\,\,\,\,if $[M]\neq[0]$},\\
\textrm{$[0]$ \,\,\,\,\,\,if $[M]=[0]$},
\end{array}\right.
$$
where $\delta^{t,op}$ is the composition of $\delta^{t}$ with the linear map $x\otimes y\mapsto y\otimes x$ and $\underline{m}^t(x,y)=x\cdot y$.

Let $\mathcal{C}^{tw}(\mathcal{A})$ be the subalgebra of the twisted Ringel-Hall algebra $\mathcal{H}^{tw}(\mathcal{A})$ generated by isomorphism classes of simple $kQ$-modules. Denoted by $\mathcal{C}_{\mathbb{Z}[v_q, v_q^{-1}]}^{tw}(\mathcal{A})$ the integral form of $\mathcal{C}^{tw}(\mathcal{A})$.

\begin{theorem}\cite{Ringel1990,Green,Xiao}
Let $\mathfrak{g}_Q$ be the Kac-Moody algebra associated to the quiver $Q$ and $U^+_{v_q}(\mathfrak{g}_Q)$ be the positive part of the quantum group $U_{v}(\mathfrak{g}_Q)$ specialized at $v=v_q$.  Then there is an isomorphism of algebras
$$
\Psi: U^+_{v_q}(\mathfrak{g}_Q)\rightarrow \mathcal{C}^{tw} _{\mathbb{Z}[v_q, v_q^{-1}]}(\mathcal{A})\bigotimes_{\mathbb{Z}[v_q, v_q^{-1}]}\mathbb{Q}(v_q)
$$
sending $E_i$ to $[S_i]$ for $i\in Q_0$.
\end{theorem}

\section{Lusztig's construction of Hall algebras via functions}
In this section, we recall Lusztig's construction of Hall algebras via functions in \cite{Lusztig1998} and compare it with Ringel-Hall algebras. Let $k=\mathbb{F}_q$  as above and $\mathbb{K}=\overline{\mathbb{F}}_q.$ Let $Q=(Q_0, Q_1, s, t)$ be a quiver. Given a dimension vector $\alpha=\sum_{i\in Q_0}\alpha_ii\in \mathbb{N}Q_0,$ define the variety
$$
\mathbb{E}_{\alpha }:=\mathbb{E}_{\alpha }(Q)=\bigoplus_{h\in Q_1}\mathrm{Hom}_{\mathbb{K}}(\mathbb{K}^{\alpha_{s(h)}}, \mathbb{K}^{\alpha_{t(h)}}).
$$
Any element $x=(x_{h})_{h\in Q_1}$ in $\mathbb{E}_{\alpha }(Q)$ defines
a representation $M(x)=(\mathbb{K}^{\alpha }, x)$ of $Q$ with
$\mathbb{K}^{\alpha }=\bigoplus_{i\in Q_0}\mathbb{K}^{\alpha_i}$. The algebraic group
$$G_{\alpha }:=G_{\alpha }(Q)=\prod_{i\in Q_0}GL(\alpha_i,\mathbb{K})$$  acts on
 $\mathbb{E}_{\alpha }$ by $(x_{h})^{g}_{h\in Q_1}=(g_{t(h)}x_{h}g_{s(h)}^{-1})_{h\in Q_1}$ for $g=(g_i)_{i\in Q_0}\in G_{\alpha }$ and
  $(x_{h})_{h\in Q_1}\in\mathbb{E}_{\alpha }.$ The isomorphism class of a $\mathbb{K}Q$-module $X$ is just the orbit of $X$.  The quotient stack $[\mathbb{E}_{\alpha }/G_{\alpha }]$ parametrizes the isomorphism classes of $\mathbb{K}Q$-modules of dimension vector $\alpha .$

Let $F$ be the Frobenius automorphism of $\mathbb{K}$, i.e., $F(x)=x^q$. The $F$-fixed subfield is just $\mathbb{F}_q$. This induces
an isomorphism $\mathbb{E}_{\alpha }\rightarrow \mathbb{E}_{\alpha }$ sending $(((x_{h})_{ij})_{d_{s(h)}\times d_{t(h)}})_{h\in Q_1}$ to $(((x_{h})^q_{ij})_{d_{s(h)}\times d_{t(h)}})_{h\in Q_1}$. We will denote all induced map by $F$ if it does not cause any confusion.

For any $\mathbb{K}Q$-module $M(x)=(\mathbb{K}^{\alpha }, x)$, set $M(x)^{[q]}=F(M(x)).$ The representation $M(x)\in \mathbb{E}_{\alpha }$ is $F$-fixed if $M(x)\cong M(x)^{[q]}.$ The last condition is equivalent to say that $M(x)$ is defined over $\mathbb{F}_q$, i.e., there exists a $kQ$-module $M_0(x)$ such that $M(x)\cong M_0(x)\otimes_{\mathbb{F}_q}\mathbb{K}$ (\cite{Hua2}). We denoted by $\mathbb{E}^F_{\alpha }$ and $G^F_{\alpha }$ the $F$-fixed subset of $\mathbb{E}_{\alpha }$ and $G_{\alpha }$ respectively. For a $kQ$-module $M\in \mathbb{E}^F_{\alpha }$, let $\mathcal{O}_M$ denote the orbit of $M$ in $\mathbb{E}_{\alpha }$ and $\mathcal{O}^F_M$ the $F$-fixed subset of $\mathcal{O}_M$.

Let $l\neq p$ be a prime number and  $\overline{\mathbb{Q}}_l$ be the algebraic closure of the field of $l$-adic numbers. Fix a square root $v_q=\sqrt{q}\in \overline{\mathbb{Q}}_l.$ Define $\mathcal{CF}^F_{\alpha }$ to be the $\overline{\mathbb{Q}}_l$-space generated by $G^F_{\alpha }$-invariant functions: $\mathbb{E}^F_{\alpha }\rightarrow \overline{\mathbb{Q}}_l.$  We will endow the vector space $\mathcal{CF}^F(Q)=\bigoplus_{\alpha }\mathcal{CF}^F_{\alpha }$ with a multiplication and a comultiplication.

As tools, we should recall two functors: the pushforward functor and the inverse image functor in \cite{Lusztig1998}. Given two finite sets $X, Y$ and a map $\phi: X\rightarrow Y$. Let $\mathcal{CF}(X)$ be the vector space of all functions $X\rightarrow \overline{\mathbb{Q}}_l$ over $X.$  Define the pushforward of $\phi$ to be
$$
\phi_{!}: \mathcal{CF}(X)\rightarrow \mathcal{CF}(Y),\,\,\phi_{!}(f)(y)=\sum_{x\in \phi^{-1}(y)}f(x)
$$
and the inverse image of $\phi$ to be
$$
\phi^{*}: \mathcal{CF}(Y)\rightarrow \mathcal{CF}(X),\,\,\phi^{*}(g)(x)=g(\phi(x)).
$$

First, We shall define the multiplication over $\mathcal{CF}^F(Q).$
Let $\mathbb{E}^{''}$ be the variety of all pairs $(x,W)$ where $x\in \mathbb{E}_{\alpha+\beta}$ and $(W, x|_W)$ is  a $\mathbb{K}Q$-submodule of $(\mathbb{K}^{\alpha+\beta}, x)$ with dimension vector $\beta$. Let $\mathbb{E}'$ be the variety of all quadruples $(x, W, \rho_1, \rho_2)$ where $(x, W)\in \mathbb{E}^{''}$ and $\rho_1: \mathbb{K}^{\alpha+\beta}/W\cong \mathbb{K}^{\alpha},$ $\rho_2: W\cong\mathbb{K}^{\beta}$ are linear isomorphisms. Consider the following diagram
$$
\xymatrix{\mathbb{E}_{\alpha}\times\mathbb{E}_{\beta}&\mathbb{E}'\ar[r]^{p_2}\ar[l]_--{p_1}&\mathbb{E}''\ar[r]^{p_3}&\mathbb{E}_{\alpha+\beta}},
$$
where $p_2, p_3$ are natural projections and $p_1(x, W, \rho_1, \rho_2)=(x',x'')$ such that
$$x'_h(\rho_1)_{s(h)}=(\rho_1)_{t(h)}x_h\,\textrm{ and }\,
x''_h(\rho_2)_{s(h)}=(\rho_2)_{t(h)}x_h$$
for any $h\in Q_1.$

The groups $G_{\alpha}\times G_{\beta}$ and $G_{\alpha+\beta}$ naturally act on $\mathbb{E}'$. The map $p_1$ is $G_{\alpha+\beta}\times G_{\alpha}\times G_{\beta}$-equivariant under the trivial action of $G_{\alpha+\beta}$ on $\mathbb{E}_{\alpha}\times\mathbb{E}_{\beta}.$ The map $p_2$ is a principal $G_{\alpha}\times G_{\beta}$-bundle.

Applying the Frobenius map $F$, we can define the above diagram over $\mathbb{F}_q$ as follows:
$$
\xymatrix{\mathbb{E}^F_{\alpha}\times\mathbb{E}^F_{\beta}&\mathbb{E}'^F\ar[r]^{p_2}\ar[l]_--{p_1}&\mathbb{E}''^{F}\ar[r]^{p_3}&\mathbb{E}^F_{\alpha+\beta}}.
$$
There is a linear map (called the induction map)
$$
\underline{m}_{\alpha,\beta}^F=\underline{m}_{\alpha,\beta}: \mathcal{CF}_{G^F_{\alpha}\times G^F_{\beta}}(\mathbb{E}^F_{\alpha}\times\mathbb{E}^F_{\beta})\rightarrow \mathcal{CF}_{ G^F_{\alpha+\beta}}(\mathbb{E}^F_{\alpha+\beta})=\mathcal{CF}^F_{\alpha+\beta}
$$
sending $g$ to $|G^F_{\alpha}\times G^F_{\beta}|^{-1}(p_3)_{!}(p_2)_{!}p^*_1(g).$
Iteratively, one can define the $r$-fold version $\underline{m}_{\alpha_1,\alpha_2,\dots,\alpha_r}^r$ of $\underline{m}_{\alpha_1,\alpha_2}$ for $r\geq 1$ by setting $\underline{m}_{\alpha_1}^1=id$, $\underline{m}_{\alpha_1,\alpha_2}^2=\underline{m}_{\alpha_1,\alpha_2}$ and $\underline{m}_{\alpha_1,\alpha_2,\dots,\alpha_{r+1}}^{r+1}=\underline{m}_{\alpha_1,\alpha_2+\alpha_3+\dots+\alpha_{r+1}}\circ (1\otimes\underline{m}_{\alpha_2,\alpha_3,\dots,\alpha_{r+1}}^r)$ for $r\geq 2.$


Now we can define the multiplication over $\mathcal{CF}^F(Q)$. For $f_{\alpha}\in \mathcal{CF}^F_{\alpha}, f_{\beta}\in \mathcal{CF}^F_{\beta}$ and $(x_1, x_2)\in \mathbb{E}_{\alpha}\times\mathbb{E}_{\beta}$, set $g(x_1, x_2)=f_{\alpha}(x_1)f_{\beta}(x_2)$. Then $g\in\mathcal{CF}_{G_{\alpha}\times G_{\beta}}(\mathbb{E}^F_{\alpha}\times\mathbb{E}^F_{\beta})$ and define the multiplication by
$$
f_{\alpha}*f_{\beta}=\underline{m}_{\alpha,\beta}(g).
$$

The following lemma is well-known from Lusztig (see \cite{Lin}).
\begin{lemma}\label{multiplication}
Given three $kQ$-modules $M, N$ and $L$, let $1_{\mathcal{O}_M}, 1_{\mathcal{O}_N}$ and $1_{\mathcal{O}_L}$ be the characteristic functions over orbits, respectively. Then $$1_{\mathcal{O}^F_M}*1_{\mathcal{O}^F_N}(L)=F_{MN}^L.$$
\end{lemma}

We now turn to define the comultiplication over $\mathcal{CF}^F(Q).$ Fix a subspace $W$ of $\mathbb{K}^{\alpha+\beta}$ with $\underline{\mathrm{dim}}W=\beta$ and linear isomorphisms $\rho_1: \mathbb{K}^{\alpha+\beta}/W\cong \mathbb{K}^{\alpha},$ $\rho_2: W\cong\mathbb{K}^{\beta}$. Let $F_{\alpha,\beta}$ be the closed subset of $\mathbb{E}_{\alpha+\beta}$ consisting of all $x\in \mathbb{E}_{\alpha+\beta}$ such that $(W, x|_W)$ is  a $\mathbb{K}Q$-submodule of $(\mathbb{K}^{\alpha+\beta}, x)$ with dimension vector $\beta$. Consider the diagram
$$
\xymatrix{\mathbb{E}_{\alpha}\times\mathbb{E}_{\beta}&F_{\alpha, \beta}\ar[l]_-{\kappa}\ar[r]^{i}&\mathbb{E}_{\alpha+\beta}},
$$
where the map $i$ is the inclusion and $\kappa(x)=p_1(x, W, \rho_1,\rho_2).$ For $(x_1,x_2)\in \mathbb{E}_{\alpha}\times\mathbb{E}_{\beta}$, the fibre $\kappa^{-1}(x_1,x_2)\cong\bigoplus_{h\in Q_1} \mathrm{Hom}_{\mathbb{K}}(\mathbb{K}^{\alpha_{s(h)}}, \mathbb{K}^{\beta_{t(h)}})$ and then $\kappa$ is a vector bundle of dimension $\sum_{h\in Q_1}\alpha_{s(h)}\beta_{t(h)}.$

Applying the Frobenius map $F$, we can define the above diagram over $\mathbb{F}_q$ as follows:
$$
\xymatrix{\mathbb{E}^F_{\alpha}\times\mathbb{E}^F_{\beta}&F^F_{\alpha, \beta}\ar[l]_-{\kappa}\ar[r]^{i}&\mathbb{E}^F_{\alpha+\beta}}.
$$
There is also a linear map (called  the restriction map)
$$
\tilde{\delta}_{\alpha,\beta}^F=\tilde{\delta}_{\alpha,\beta}: \mathcal{CF}_{G^F_{\alpha+\beta}}(\mathbb{E}^F_{\alpha+\beta})\rightarrow\mathcal{CF}_{G^F_{\alpha}\times G^F_{\beta}}(\mathbb{E}^F_{\alpha}\times\mathbb{E}^F_{\beta})
$$
sending $f\in \mathcal{CF}_{G_{\alpha+\beta}}(\mathbb{E}^F_{\alpha+\beta})$ to $\kappa_{!}i^*(f).$ It is clear that there is an isomorphism
$$
\mathcal{CF}^F_{\alpha}\otimes\mathcal{CF}^F_{\beta}\cong \mathcal{CF}_{G^F_{\alpha}\times G^F_{\beta}}(\mathbb{E}^F_{\alpha}\times\mathbb{E}^F_{\beta})
$$
by sending $f\otimes g$ to the function mapping $(x,y)\in \mathbb{E}^F_{\alpha}\times\mathbb{E}^F_{\beta}$ to $f(x)g(y).$ Hence, we can write $$\tilde{\delta}_{\alpha,\beta}^F=\tilde{\delta}_{\alpha,\beta}:\mathcal{CF}^F_{\alpha+\beta}\rightarrow \mathcal{CF}^F_{\alpha}\otimes\mathcal{CF}^F_{\beta}$$
and define the comultiplication $\tilde{\delta}^F=\tilde{\delta}$ over $\mathcal{CF}^F(Q)$, i.e., for $f\in \mathcal{CF}^F_{\gamma}$ and $\alpha+\beta=\gamma$, $$\tilde{\delta}(f)=\sum_{\alpha,\beta;\alpha+\beta=\gamma}\tilde{\delta}_{\alpha,\beta}(f).$$
Iteratively, we can define $\tilde{\delta}_{\alpha_1,\alpha_2,\dots,\alpha_r}^r$ for $r\geq 1$ by setting $\tilde{\delta}_{\alpha_1,\alpha_2}^1=\tilde{\delta}_{\alpha_1,\alpha_2}$ and $\tilde{\delta}_{\alpha_1,\alpha_2,\dots,\alpha_{r+1}}^{r+1}=(1\otimes \cdots \otimes \tilde{\delta}_{\alpha_r,\alpha_{r+1}})\circ \tilde{\delta}_{\alpha_1,\alpha_2,\dots,\alpha_r+\alpha_{r+1}}^r$ for $r\geq 1.$


For $M, N$ and $L$ in $\mathcal{A}=\mathrm{Rep}_{k}Q$, we set $$D_{L}^{MN}=\tilde{\delta}(1_{\mathcal{O}^F_L})(M, N)=\kappa_!i^*(1_{\mathcal{O}^F_L})(M, N).$$
In order to compare this  comultiplication $\tilde{\delta}$ with the comultiplication of Ringel-Hall algebras, we define the twist of $\tilde{\delta}_{\alpha.\beta}$ by $\delta_{\alpha,\beta}=q^{-\sum_{i\in Q_0}\alpha_i\beta_i}\tilde{\delta}_{\alpha,\beta}$ and $\delta$ in the same way.
\begin{lemma}\label{comultiplication}
With the notations in Lemma \ref{multiplication} and $\underline{\mathrm{dim}}M=\alpha, \underline{\mathrm{dim}}N=\beta$, we have $$\delta_{\alpha, \beta}(1_{\mathcal{O}^F_L})(M, N)=h_{L}^{MN}.$$
\end{lemma}
\begin{proof}
Suppose $M=(\mathbb{K}^{\alpha}, x_1)$ and $N=(\mathbb{K}^{\beta}, x_2)$. The linear isomorphisms $\rho_1, \rho_2$ induce the module structures of $\mathbb{K}^{\alpha+\beta}/W$ and $W$, denoted by $(\mathbb{K}^{\alpha+\beta}/W, y_{\mathbb{K}^{\alpha+\beta}/W})$ and $(W, y_W)$ respectively. Consider the set
$$S=\{x\in \mathbb{E}^{\alpha+\beta}\mid (W, x|_W)=(W, y_W),$$$$
(\mathbb{K}^{\alpha+\beta}/W, x|_{\mathbb{K}^{\alpha+\beta}/W})=(\mathbb{K}^{\alpha+\beta}/W, y_{\mathbb{K}^{\alpha+\beta}/W}),
(\mathbb{K^{\alpha+\beta}}, x)\cong L\}.$$
Fix a decomposition of the vector space $\mathbb{K}^{\alpha+\beta}=W\oplus \mathbb{K}^{\alpha+\beta}/W.$ Then
$$
S=\left\{x=\left(
        \begin{array}{cc}
          (y_W)_h & d(h) \\
          0 & (y_{\mathbb{K}^{\alpha+\beta}/W})_h \\
        \end{array}
      \right)_{h\in Q_1}\mid d(h)\in \mathrm{Hom}_{\mathbb{K}}(\mathbb{K}^{\alpha_{s(h)}}, \mathbb{K}^{\beta_{t(h)}}),  (\mathbb{K^{\alpha+\beta}}, x)\cong L
\right\}.
$$
Set $D(\alpha, \beta)=\bigoplus_{h\in Q_1} \mathrm{Hom}_{\mathbb{K}}(\mathbb{K}^{\alpha_{s(h)}}, \mathbb{K}^{\beta_{t(h)}})$. Applying the Frobenius map $F$,  we have the following long exact sequence (see \cite{CB})
$$
\xymatrix{0\ar[r]&\mathrm{Hom}_{kQ}(M, N)\ar[r]&\oplus_{i\in Q_0}\mathrm{Hom}_{k}(k^{\alpha_i}, k^{\beta_i})\ar[r]&}
$$
$$
\xymatrix{\ar[r]&D^F(\alpha, \beta)\ar[r]^-{\pi}&\mathrm{Ext}_{kQ}^1(M, N)\ar[r]&0.}
$$
We denote by $D^F(\alpha, \beta)_{L}$ the inverse image of $\mathrm{Ext}^1_{kQ}(M, N)_L$ under the map $\pi.$ Then $D_{L}^{MN}=|D^F(\alpha, \beta)_L|=q^{\sum_{i\in Q_0}\alpha_i\beta_i}h_{L}^{MN}.$
By definition, $\tilde{\delta}_{\alpha,\beta}(1_{\mathcal{O}^F_L})(M, N)=|D^F(\alpha, \beta)_L|=D_{L}^{MN}.$ This completes the proof.
\end{proof}

In order to compare these with Lusztig's construction, we consider the subalgebra of $\mathcal{CF}^F(Q)$ generated by $1_{S_i}=1_{\mathcal{O}^F_{S_i}}$ for all $i\in Q_0$, denoted by $\mathcal{F}^F(Q).$ The subalgebra $\mathcal{F}^F(Q)$ has the following decomposition of weight spaces: $$\mathcal{F}^F(Q)=\bigoplus_{\alpha }\mathcal{F}^F_{\alpha }.$$

\begin{lemma}\label{coincide}
Give a sequence $\mathbf{i}=(i_1, i_2, \cdots, i_m)$ in $Q_0$  such that $i_j\neq i_k$ for $j\neq k\in \{1, 2, \cdots, m\}$ and let $f=1_{S_{i_1}}*1_{S_{i_2}}*\cdots *1_{S_{i_m}}\in \mathcal{CF}^F(Q)$. Then $\tilde{\delta}(f)=\delta(f).$
\end{lemma}

The Riedtmann-Peng formula can be reformulated to the following form, which generalizes \cite[Lemma 1.13]{Lusztig1998} from $\mathcal{F}^F(Q)$ to $\mathcal{CF}^F(Q).$
\begin{Prop}\label{RP}
Let $f_i\in \mathcal{CF}^F_{\alpha _i}$ for $i=1, 2$ and $g\in \mathcal{CF}^F_{\alpha }$ for $\alpha =\alpha _1+\alpha _2.$ Then
$$
|G^F_{\alpha }|\sum_{x, y}f_1(x)f_2(y)\delta_{\alpha _1,\alpha _2}(g)(x,y)=|G^F_{\alpha _1}\times G^F_{\alpha _2}|\sum_{z}f_1*f_2(z)g(z)
$$
where $x\in \mathbb{E}^F_{\alpha _1}, y\in \mathbb{E}^F_{\alpha _2}$ and $z\in \mathbb{E}^F_{\alpha _1}.$
\end{Prop}
\begin{proof}
Given a dimension vector $\alpha $, take $f\in \mathcal{CF}_{G_{\alpha}}(\mathbb{E}^F_{\alpha })$, then $f=\sum_{i=1}^sa_i 1_{\mathcal{O}^F_{M_i}}$ for some $a_i\in \overline{\mathbb{Q}}_l$, $s\in \mathbb{Z}$ and $kQ$-modules $M_1, \cdots, M_s.$  Without loss of generality, we may assume that $f_1=1_{\mathcal{O}^F_M}$, $f_2=1_{\mathcal{O}^F_N}$ and $g=1_{\mathcal{O}^F_L}$ for some $kQ$-modules $M, N$ and $L.$ Following Lemma \ref{multiplication} and \ref{comultiplication}, the left side of the equation is equal to
$$
|G^F_{\alpha }|\cdot|\mathcal{O}^F_M|\cdot|\mathcal{O}^F_N|\cdot h_{MN}^L
$$
and the right side of the equation is equal to
$$
|G^F_{\alpha _1}|\cdot|G^F_{\alpha _2}|\cdot|\mathcal{O}^F_L|\cdot F_{MN}^L.
$$
Using $a_L=|G^F_{\alpha }|/|\mathcal{O}^F_L|$ and the Riedtmann-Peng formula, we prove the proposition.
\end{proof}

By definition, $\mathrm{dim}_kG^F_{\alpha }=\sum_{i\in Q_0}\alpha^2_i$ and then we obtain the following lemma (\cite[Section 1.2]{Schiffmann2}).
\begin{lemma}
With the above notations, we have $$\frac{1}{2}(\mathrm{dim}_kG^F_{\alpha }-\mathrm{dim}_k G^F_{\alpha _1}-\mathrm{dim}_k G^F_{\alpha _2})=\sum_{i\in Q_0}(\alpha _1)_i(\alpha _2)_i.$$
\end{lemma}
Hence, the equation in Proposition \ref{RP} can also be written as
$$
\frac{|G^F_{\alpha }|}{|\mathfrak{g}^F_{\alpha }|}\cdot\sum_{x, y}f_1(x)f_2(y)\tilde{\delta}_{\alpha _1,\alpha _2}(g)(x,y)=\frac{|G^F_{\alpha _1}|}{|\mathfrak{g}^F_{\alpha _1}|}\frac{|G^F_{\alpha _2}|}{|\mathfrak{g}^F_{\alpha _2}|}\sum_{z}f_1*f_2(z)g(z)
$$
by substituting $\widetilde{\delta}$ for $\delta$, where $\mathfrak{g}^F_{\alpha }$, $\mathfrak{g}^F_{\alpha _1}$ and $\mathfrak{g}^F_{\alpha _2}$ are the Lie algebras of $G^F_{\alpha }$, $G^F_{\alpha _1}$ and $G^F_{\alpha _2}$, respectively.

Then, we shall consider the relation between the induction map and the restriction map.

Fix dimension vectors $\alpha, \beta, \alpha', \beta'$ with $\alpha+\beta=\alpha'+\beta'=\gamma $. Let $\mathcal{N}$ be the set of quadruples $\lambda=(\alpha_1, \alpha_2, \beta_1, \beta_2)$ of dimension vectors such that $\alpha=\alpha_1+\alpha_2, \beta=\beta_1+\beta_2, \alpha'=\alpha_1+\beta_1$ and $\beta'=\alpha_2+\beta_2$. Consider the following diagram
\begin{equation}\label{commutative_diagram1}
\xymatrix{\mathbb{E}^F_{\alpha}\times \mathbb{E}^F_{\beta}&\mathbb{E}'^F_{\alpha,\beta}\ar[l]_-{p_1}\ar[r]^{p_2}&\mathbb{E}^{''F}_{\alpha,\beta}\ar[r]^{p_3}&\mathbb{E}^F_{\gamma }\\
\coprod_{\lambda\in \mathcal{N}}F^F_{\lambda}\ar[u]^{i'}\ar[d]_{\kappa'} &&&F^F_{\alpha',\beta'}\ar[d]_{\kappa}\ar[u]^{i}\\
\coprod_{\lambda\in \mathcal{N}}E^F(\lambda)&\coprod_{\lambda\in \mathcal{N}}\mathbb{E}'^F_{\lambda}\ar[l]_-{p'_1}\ar[r]^{p'_2}
&\coprod_{\lambda\in \mathcal{N}}\mathbb{E}^{''F}(\lambda)\ar[r]^{p'_3}&\mathbb{E}^F_{\alpha'}\times \mathbb{E}^F_{\beta'}}
\end{equation}
where $\mathbb{E}^F(\lambda)=\mathbb{E}^F_{\alpha_1}\times\mathbb{E}^F_{\alpha_2}\times\mathbb{E}^F_{\beta_1}\times\mathbb{E}^F_{\beta_2}$, $\mathbb{E}'^F(\lambda)=\mathbb{E}'^F_{\alpha_1,\beta_1}\times\mathbb{E}'^F_{\alpha_2,\beta_2}$ and $\mathbb{E}''^{F}(\lambda)=\mathbb{E}''^{F}_{\alpha_1,\beta_1}\times\mathbb{E}''^{F}_{\alpha_2,\beta_2}$ for $\lambda=(\alpha_1, \alpha_2, \beta_1, \beta_2).$ This induces the maps between $F$-fixed subsets and then the maps between vector spaces of functions as follows:
\begin{equation}\label{commutative_diagram}\xymatrix{\mathcal{CF}^{F}_{\alpha}\times \mathcal{CF}^F_{\beta}\ar[r]^{\underline{m}_{\alpha,\beta}}\ar[d]^{\delta}&\mathcal{CF}^F_{\gamma }\ar[d]^{\delta_{\alpha',\beta'}}\\\mathcal{CF}^F(\prod_{\lambda\in \mathcal{N}}\mathbb{E}_{\lambda})\ar[r]^{\underline{m}}&\mathcal{CF}^{F}_{\alpha'}\times \mathcal{CF}^F_{\beta'}.}
\end{equation}

Consider the top and right of Diagram (\ref{commutative_diagram1}).
Set
$$
C_{\alpha,\beta, \alpha',\beta'}^{'F}=\{(x, W, \rho_1, \rho_2)\in \mathbb{E}'^F_{\alpha,\beta}\mid x\in F^F_{\alpha',\beta'}\}$$
and
$$C_{\alpha,\beta, \alpha',\beta'}^{''F}=\{(x, W)\in \mathbb{E}^{''F}_{\alpha,\beta}\mid x\in F^F_{\alpha',\beta'}\}.
$$
The sets can be illustrated by the following diagram:
$$
\xymatrix{&W'\ar[d]&\\W\ar[r]&(\mathbb{K}^{\gamma }, x)\ar[r]\ar[d]&(\mathbb{K}^{\gamma }, x)/W\\
&(\mathbb{K}^{\gamma }, x)/W'&}
$$
where $(x, W')\in \mathbb{E}''^{F}_{\alpha', \beta'}.$
Consider the following diagram
$$
\xymatrix{\mathbb{E}^F_{\alpha}\times \mathbb{E}^F_{\beta}&C_{\alpha,\beta, \alpha',\beta'}^{'F}\ar[l]_-{p}\ar[r]^{q}&C_{\alpha,\beta, \alpha',\beta'}^{''F}\ar[r]^{r}&F^F_{\alpha',\beta'}\ar[r]^-{\kappa}&\mathbb{E}^F_{\alpha'}\times \mathbb{E}^F_{\beta'}}.
$$
Then by definition, we have
\begin{equation}\label{delta_m}
\delta_{\alpha', \beta'}\underline{m}_{\alpha, \beta}=|G^F_{\alpha}\times G^F_{\beta}|^{-1}q^{-\sum_{i\in Q_0\alpha'_i\beta'_i}}(\kappa)_{!}(r)_{!}(q)_{!}p^*.
\end{equation}

Consider the left and bottom of Diagram (\ref{commutative_diagram1}). Set
$$
S_{\lambda}^{'F}=\{(x_{\alpha}, x_{\beta}, x_{\alpha'}, x_{\beta'}, W_1, W_2, \rho_{11}, \rho_{12}, \rho_{21}, \rho_{22})\mid (x_{\alpha'}, W_1, \rho_{11}, \rho_{12})\in \mathbb{E}^{'F}_{\alpha_2,\beta_2},$$$$(x_{\beta'}, W_2, \rho_{21}, \rho_{22})\in \mathbb{E}^{'F}_{\alpha_1,\beta_1}, (x_{\beta}, W_2)\in \mathbb{E}''^{F}_{\beta_1,\beta_2}, (\mathbb{K}^{\beta}, x_{\beta})/W_2\cong (W_1, x_{\alpha'\mid_{W_1}}),$$$$\exists W_3, (x_{\alpha}, W_3)\cong (\mathbb{K}^{\beta'}, x_{\beta'})/W_2, (\mathbb{K}^{\alpha}, x_{\alpha})/W_3\cong (\mathbb{K}^{\alpha'}, x_{\alpha'})/W_1
\}
$$
where $\lambda=(\alpha_1, \alpha_2, \beta_1, \beta_2)$ and $W_1, W_2$ and $W_3$ are graded  vector spaces of dimension vectors $\beta_2, \beta_1$ and $\alpha_1$, respectively.
The set can be illustrated by the following diagram:
$$
\xymatrix{W_2\ar[r]\ar[d]&(\mathbb{K}^{\beta'}, x_{\beta'})\ar[r]&W_3\ar[d]\\
(\mathbb{K}^{\beta}, x_{\beta})\ar[d]&&(\mathbb{K}^{\alpha}, x_{\alpha})\ar[d]\\
W_1\ar[r]&(\mathbb{K}^{\alpha'}, x_{\alpha'})\ar[r]&(\mathbb{K}^{\alpha}, x_{\alpha})/W_3\cong (\mathbb{K}^{\alpha'}, x_{\alpha'})/W_1.}
$$
Set
$$
S_{\lambda}^{''F}=\{(x_{\alpha}, x_{\beta}, x_{\alpha'}, x_{\beta'}, W_1, W_2)\mid \exists  \rho_{11}, \rho_{12}, \rho_{21}, \rho_{22}, $$$$(x_{\alpha}, x_{\beta}, x_{\alpha'}, x_{\beta'}, W_1, W_2, \rho_{11}, \rho_{12}, \rho_{21}, \rho_{22})\in S_{\lambda}^{'F}\}.
$$
Then there is a projection $S_{\lambda}^{'F}\rightarrow S_{\lambda}^{''F}$ which is a principal $G_{\alpha_1}\times G_{\alpha_2}\times G_{\beta_1}\times G_{\beta_2}$-bundle.
We also have the following diagram
$$
\xymatrix{\mathbb{E}^F_{\alpha}\times \mathbb{E}^F_{\beta}&\coprod_{\lambda\in \mathcal{N}}F^F_{\lambda}\ar[l]_-{i'}&\coprod_{\lambda\in \mathcal{N}}S_{\lambda}^{'F}\ar[l]_-{p'}\ar[r]^-{q'}&\coprod_{\lambda\in \mathcal{N}}S_{\lambda}^{''F}\ar[r]^-{r'}&\mathbb{E}^F_{\alpha'}\times \mathbb{E}^F_{\beta'}}.
$$
Then we have
\begin{equation}\label{m_delta}
\underline{m}\delta=|G^F_{\alpha_1}\times G^F_{\alpha_2}\times G^F_{\beta_1}\times G^F_{\beta2}|^{-1}q^{-\langle \alpha_2, \beta_1\rangle-\sum_{i\in Q_0}[(\beta_1)_i(\beta_2)_i+(\alpha_1)_i(\alpha_2)_i]}(r')_{!}(q')_{!}(p')^*(i')^*.
\end{equation}

By (\ref{delta_m}) and (\ref{m_delta}), Diagram (\ref{commutative_diagram1}) can be rewrote as:
$$
\xymatrix{\mathbb{E}^F_{\alpha}\times \mathbb{E}^F_{\beta}& C'^F\ar[l]-_p\ar[d]^{\kappa r q}\\
\coprod_{\lambda\in \mathcal{N}}S'^F_{\lambda}\ar[u]^{p'i'}\ar[r]^-{r'q'}&\mathbb{E}^F_{\alpha'}\times \mathbb{E}^F_{\beta'}.}
$$

By Lemma \ref{multiplication} and \ref{comultiplication}, one can check the following two lemmas directly.
\begin{lemma}\label{lemma:delta_m}
For $M\in \mathbb{E}^F_{\alpha}, N\in \mathbb{E}^F_{\beta}, M'\in \mathbb{E}^F_{\alpha'}, N'\in \mathbb{E}^F_{\beta'}$, we have
$$
\delta_{\alpha', \beta'}\underline{m}_{\alpha, \beta}(1_{\mo^F_M}, 1_{\mo^F_N})(M', N')=\sum_{[L]\in \mathbb{E}^F_{\alpha }/G^F_{\alpha }}F^L_{MN}h_{L}^{M'N'}.
$$
\end{lemma}
\begin{lemma}\label{lemma:m_delta}
For $M\in \mathbb{E}^F_{\alpha}, N\in \mathbb{E}^F_{\beta}, M'\in \mathbb{E}^F_{\alpha'}, N'\in \mathbb{E}^F_{\beta'}$, we have
$$
\underline{m}\delta(1_{\mo^F_M}, 1_{\mo^F_N})(M', N')=\sum_{[X],[Y_1],[Y_2],[Z]} q^{-\langle\mathrm{\underline{\dim}}X,\mathrm{\underline{\dim}}Z\rangle}F^{M'}_{XY_2}F^{N'}_{Y_1Z}h_{M}^{XY_1}h_{N}^{Y_2Z},
$$
where $[X]\in \mathbb{E}^F_{\alpha_2}/G^F_{\alpha_2},[Y_1]\in \mathbb{E}^F_{\alpha_1}/G^F_{\alpha_1},[Y_2]\in \mathbb{E}^F_{\beta_2}/G^F_{\beta_2},[Z]\in \mathbb{E}^F_{\beta_1}/G^F_{\beta_1}.$
\end{lemma}

In order to study the relation between $\underline{m}\delta$ and $\delta_{\alpha', \beta'}\underline{m}_{\alpha, \beta}$, we refer the reformulation of the proof of Green's theorem in \cite{Schiffmann1}.

First, we shall count the set of crossings with the group action. More precisely, fix $M, N, M', N'$ and consider the diagram
$$
\xymatrix{
& & 0 \ar[d] & &\\
& & N' \ar[d]^-{a'} & & \\
0 \ar[r] & N \ar[r]^-{a} & L\ar[r]^{b} \ar[d]^-{b'} & M \ar[r] & 0\\
& & M' \ar[d] & &\\
& & 0 & &}
$$
Set $$Q=\{(a, b, a', b')\mid a, b,a',b' \mbox{ as in the above crossing}\}.$$
By calculation, $|Q|=\sum_{[L]\in \mathbb{E}^F_{\alpha }/G^F_{\alpha }}F^L_{MN}h_{L}^{M'N'}|G^F_{\alpha }||G^F_{\alpha}||G^F_{\beta}|.$
Consider the natural action of $G^F_{\alpha }$ on $Q$, and the orbit space is denoted by $\widetilde{Q}.$ The fibre of the map $Q\rightarrow \widetilde{Q}$ has cardinality $\frac{|G^F_{\alpha }|}{|\mathrm{Hom}(\mathrm{Coker}b'a, \mathrm{Ker}b'a)|}.$

Next, we shall count the squares with the group action. More precisely, consider the diagram
$$
\xymatrix{ & 0 \ar[d] &&&&  0\ar[d] && \\
0 \ar[r] & Z \ar[rr]^-{e_1} \ar[dd]^-{u'} && N'\ar[rr]^-{e_2} && Y_1 \ar[r] \ar[dd]^-{x}& 0\\
&&&&&&\\
& N \ar[dd]^-{v'} && && M \ar[dd]^-{y} &\\
&&&&&&\\
0\ar[r] &  Y_2 \ar[d] \ar[rr]^-{e_3} && M' \ar[rr]^-{e_4} && X \ar[r] \ar[d] & 0\\
& 0 && && 0 &}
$$
Set $$\mo=\{(e_1,e_2,e_3,e_4,u',v',x,y)\mid \mbox{ all morphisms occur in the above diagram}\}.$$
The group $G^F_{\alpha_1}\times G^F_{\alpha_2}\times G^F_{\beta_1}\times G^F_{\beta2}$ freely acts on $\mo$ with orbit space $\widetilde{\mo}$. Note that
$$
|\widetilde{\mo}|=\sum_{[X],[Y_1],[Y_2],[Z]} F^{M'}_{XY_2}F^{N'}_{Y_1Z}h_{M}^{XY_1}h_{N}^{Y_2Z}|G^F_{\alpha}||G^F_{\beta}|.
$$

There is a canonical map $\widetilde{f}: \widetilde{Q}\rightarrow \widetilde{\mo}$. And the cardinality of the fibre of this map is $|\mathrm{Ext}^1(X, Z)|.$
Hence, we get
$$
a_{M_1}a_{M_2}a_{N_1}a_{N_2}\sum_{[L]}F_{M_1N_1}^LF_{M_2N_2}^La^{-1}_L$$$$=\sum_{[X],[Y_1],[Y_2],[Z]}\frac{|\mathrm{Ext}^1_{\mathcal{A}}(X,Z)|}{|\mathrm{Hom}_{\mathcal{A}}(X, Z)|}F_{XY_1}^{M_1}F_{XY_2}^{M_2}F_{Y_2Z}^{N_1}F_{Y_1Z}^{N_2}a_{X}a_{Y_1}a_{Y_2}a_{Z},
$$
which is the Green formula.

Applying the Riedtmann-Peng formula, we have the identity
$$
\sum_{[L]\in \mathbb{E}^F_{\alpha }/G^F_{\alpha }}F^L_{MN}h_{L}^{M'N'}=\sum_{[X],[Y_1],[Y_2],[Z]} q^{-\langle\mathrm{\underline{\dim}}X,\mathrm{\underline{\dim}}Z\rangle}
F^{M'}_{XY_2}F^{N'}_{Y_1Z}h_{M}^{XY_1}h_{N}^{Y_2Z}.
$$

The right side of the identity in Lemma \ref{lemma:m_delta} is the left side of this identity
and the right side of the identity in Lemma \ref{lemma:delta_m} is the right side of this identity.
Hence, we have the following theorem.

\begin{theorem}\label{green-geometry}
With the above notations, Diagram (\ref{commutative_diagram}) is commutative, i.e., $$\delta_{\alpha', \beta'}\underline{m}_{\alpha, \beta}=\underline{m}\delta.$$
\end{theorem}

This theorem can be viewed as the geometric analog of Green's theorem.

As in Ringel-Hall algebras, we can define the analogue $\sigma: \mathcal{CF}^F(Q)\rightarrow \mathcal{CF}^F(Q)$ of the antipode by setting
$$\sigma(f)=\sum_{r\geq 1\in \mathbb{Z}}(-1)^r\sum_{\alpha_1,\cdots,\alpha_r\neq 0}\underline{m}_{\alpha_1, \cdots,\alpha_r}^r\circ \delta_{\alpha_1,\cdots, \alpha_r}^{r}(f)$$ for $f\neq 1_0\in \mathcal{CF}^F(Q)$, where $1_0$ is the constant function on $\mathbb{E}_0$, which is the unit element in $\mathcal{CF}^{F}(Q)$.

In order to compare Lusztig's Hall algebras with twisted Ringel-Hall algebras, we twist $\mathcal{CF}^F(Q)$ by setting $\underline{m}_{\alpha, \beta}^t=v^{\langle\alpha, \beta\rangle}\underline{m}_{\alpha, \beta}$, $\delta^t_{\alpha, \beta}=v^{\langle\alpha, \beta\rangle}\delta_{\alpha,\beta}$ and $$\sigma^t(f)=\sum_{r\geq 1\in \mathbb{Z}}(-1)^r\sum_{\alpha_1,\cdots,\alpha_r\neq 0}\underline{m}_{\alpha_1,\cdots, \alpha_r}^{t,r}\circ \delta_{\alpha_1,\cdots, \alpha_r}^{t, r}(f).$$ We denote the twisted version by $\mathcal{CF}^{F,tw}(Q).$ Note that we use the same notations $\delta^t$ and $\sigma^t$ in $\mathcal{H}^{tw}(\mathcal{A})$ and $\mathcal{CF}^{F,tw}(Q)$ for convenience. With the context, it should not cause the confusion.

On the relations between $\sigma^t$ with $\underline{m}^t$ and $\delta^t$, we have
$$\sigma^t(f\cdot g)=\sigma^t(g)\cdot\sigma^t(f), \textrm{ for any $f,g\in\mathcal{CF}^{F,tw}(Q)$},$$
$$\delta^t(\sigma^t(f))=(\sigma^t\otimes\sigma^t){\delta^{t,op}}(f), \textrm{ for any $f\in\mathcal{CF}^{F,tw}(Q)$},$$
and
\begin{equation}\label{antipode_multi_comulti}
\underline{m}^t(\sigma^t\otimes 1)\delta^t(f)=\underline{m}^t(1\otimes \sigma^t)\delta^t(f)
=\left\{
\begin{array}{c}
\textrm{$0$ \,\,\,\,\,\,if $f\neq1_{0}\in\mathcal{CF}^{F,tw}(Q)$},\\
\textrm{$1_{0}$ \,\,\,\,\,\,if $f=1_{0}\in\mathcal{CF}^{F,tw}(Q)$},
\end{array}\right.
\end{equation}
where $\delta^{t,op}$ is the composition of $\delta^{t}$ with the linear map $x\otimes y\mapsto y\otimes x$.

By applying Lemma \ref{multiplication}, \ref{comultiplication} and Theorem \ref{green-geometry}, we have the following theorem.
\begin{theorem}
Let $\mathcal{A}=\mathrm{Rep}_{k}Q$. Fix an isomorphism $\iota:\overline{\mathbb{Q}}_l\rightarrow\mathbb{C}$. There is an isomorphism of algebras
\begin{eqnarray*}
\Phi: \mathcal{CF}^{F,tw}(Q)&\rightarrow&\mathcal{H}^{tw}(\mathcal{A})\\
1_{\mathcal{O}^F_M}&\mapsto&[M]
\end{eqnarray*}
satisfying $\delta^{t}\circ\Phi=(\Phi\otimes\Phi)\circ\delta^t$ and $\sigma^t\circ\Phi=\Phi\circ\sigma^t$.
\end{theorem}

\section{The categorification of Ringel-Hall algebras}

Let $\mathbb{F}_q$ be a finite field with $q$ elements. In the following, $\mathbb{K}$ is an algebraic closure of $\mathbb{F}_q$.

Let $X$ be a scheme of finite type over $\mathbb{K}$. We say that $X$ has an $\mathbb{F}_q$-structure if there exists a variety $X_0$ over $\mathbb{F}_q$ such that $X=X_0\times_{Spec({\mathbb{F}_q})}Spec(\mathbb{K})$. Let $F_{X_0}: X_0\rightarrow X_0$ be the Frobenius morphism. It can be extended to the morphism $F_{X}:X\rightarrow X.$ Let $X^F$ be the set of closed points of $X$ fixed by $F$, i.e., the set of $\mathbb{F}_q$-rational points.
For any $n\in\mathbb{N}$, let $X^{F^n}$ be the set of closed points of $X$ fixed by $F^n$. Note that $X^{F^1}=X^F$.

Denote by $\mathcal{D}^b(X)=\mathcal{D}^b(X, \overline{\mathbb{Q}}_l)$ the bounded derived category of $\overline{\mathbb{Q}}_l$-constructible complexes on $X$.

The morphism $F_{X}: X\rightarrow X$ naturally induces a functor $F_{X}^*: \mathcal{D}^b(X)\rightarrow \mathcal{D}^b(X)$. A Weil complex is a pair $(\mathcal{F}, j)$ such that $\mathcal{F}\in\mathcal{D}^b(X)$ and $j: F_X^*(\mathcal{F})\rightarrow \mathcal{F}$ is an isomorphism.

Fix an isomorphism $\iota:\overline{\mathbb{Q}}_l\rightarrow \mathbb{C}.$ We refer to \cite{KW} for the definitions of $\iota$-pure  and $\iota$-mixed complexes.
Let $x\in X^F$ be a closed point. For any Weil sheaf $\mathcal{F}$ on $X$, the isomorphism $j: F_X^*(\mathcal{F})\rightarrow \mathcal{F}$ induces
an automorphism $$j_x:\mathcal{F}_{\mid x}\rightarrow\mathcal{F}_{\mid x}.$$
For any $w\in\mathbb{R}$, the weil sheaf $\mathcal{F}$ on $X$ is called $\iota$-pure of weight $w$ if all eigenvalues $\lambda$ of the automorphism
$j^n_x$ satisfies that $|\iota(\lambda)|=(q^n)^{\frac{w}{2}}$ for any $n$ and any $x\in X^{F^n}$.
In this case, we denote $\mbox{Weight}(\mathcal{F})=w$. The sheaf $\mathcal{F}$ is called $\iota$-mixed if it admits a finite filtration of successive quotient which is $\iota$-pure.

For any Weil complex $\mathcal{F}$, $\mathcal{F}$ is called $\iota$-mixed if the cohomology sheaves $\mathcal{H}^i(\mathcal{F})$ are mixed.
Let $\mathcal{D}_w^b(X)$ be the triangulated subcategory of $\mathcal{D}^b(X)$ of $\iota$-mixed Weil complexes and $K_w(X)$ be the Grothendieck group of $\mathcal{D}_w^b(X)$.

Let $x\in X^F$ be a closed point. Given a Weil complex $\mathcal{F}=(\mathcal{F},j)$
in $\mathcal{D}_w^b(X)$,
we get automorphisms $$F_{i,x}:\mathcal{H}^i(\mathcal{F})_{\mid x}\rightarrow\mathcal{H}^i(\mathcal{F})_{\mid x}.$$
One can define a $F$-invariant function $\chi^F_{\mathcal{F}}: X^F\rightarrow \overline{\mathbb{Q}}_l$ via defining
$$\chi^F_{\mathcal{F}}(x)=\sum_{i}(-1)^itr(F_{i,x},\mathcal{H}^i(\mathcal{F})_{\mid x})=\sum_{i}(-1)^itr(F_{i,x}).$$
Similarly, one can define $\chi^{F^n}_{\mathcal{F}}: X^{F^n}\rightarrow \overline{\mathbb{Q}}_l$
via defining
$$\chi^{F^n}_{\mathcal{F}}(x)=\sum_{i}(-1)^itr(F^n_{i,x},\mathcal{H}^i(\mathcal{F})_{\mid x})=\sum_{i}(-1)^itr(F^n_{i,x}).$$
In particular, $\chi^{F^1}=\chi^F.$

\begin{theorem}\cite[Theorem 12.1]{KW}\label{tool}
Let $X$ be as above. Then $\chi^F$ satisfies the following properties.
\begin{enumerate}
  \item Let $\mathcal{K}\rightarrow\mathcal{L}\rightarrow\mathcal{M}\rightarrow\mathcal{K}[1]$ be a distinguished triangle in $\mathcal{D}_w^b(X)$. Then $\chi^F_{\mathcal{K}}+\chi^F_\mathcal{M}=\chi^F_\mathcal{L}.$
  \item Let $g: X\rightarrow Y$ be a morphism. Then for $\mathcal{K}\in \mathcal{D}_w^b(X)$ and $\mathcal{L}\in \mathcal{D}_w^b(Y),$ we have $\chi^F_{Rg_{!}\mathcal{K}}=g_{!}(\chi^F_{\mathcal{K}})$ and $\chi^F_{g^*\mathcal{L}}=g^*(\chi^F_{\mathcal{L}}).$
  \item For $\mathcal{K}\in \mathcal{D}_w^b(X)$, we have $\chi^F_{\mathcal{K}[d]}=(-1)^d\chi^F_{\mathcal{K}}$ and $\chi^F_{\mathcal{K}(n)}=q^{-n}\chi^F_{\mathcal{K}}$.
\end{enumerate}
\end{theorem}

By Theorem \ref{tool}(1), the function $\chi^{F^n}_{\mathcal{F}}$ only depends on the isomorphism class of $\mathcal{F}$ in $\mathcal{D}_w^b(X)$. Let $\mathcal{CF}(X^{F^n})$ be the vector space of all functions $X^{F^n}\rightarrow \overline{\mathbb{Q}}_l.$ Hence, we obtain a map $\chi^{F^n}: K_w(X)\rightarrow \mathcal{CF}(X^{F^n}).$

Let $G$ be an algebraic group over $\mathbb{K}$ and $X$ be a scheme of finite type over $\mathbb{K}$ together with a $G$-action.
Assume that $X$ and $G$ have $\mathbb{F}_q$-structures and $X=X_0\times_{Spec({\mathbb{F}_q})}Spec(\mathbb{K})$, $G=G_0\times_{Spec({\mathbb{F}_q})}Spec(\mathbb{K})$. Let $F_{G_0}: G_0\rightarrow G_0$ be the Frobenius morphism. It can be extended to the morphism $F_{G}:G\rightarrow G.$
Denote by $\mathcal{D}_G^b(X)=\mathcal{D}_G^b(X, \overline{\mathbb{Q}}_l)$ the $G$-equivariant bounded derived category of $\overline{\mathbb{Q}}_l$-constructible complexes on $X$ and $\mathcal{D}_{G,w}^b(X)$ the subcategory of $\mathcal{D}_G^b(X)$ consisting of $\iota$-mixed Weil complexes.
Let $K_{G,w}(X)$ be the Grothendieck group of $\mathcal{D}_{G,w}^b(X)$.

Assume that we have the following commutative diagram
$$\xymatrix{
  G\times X \ar[d]_{F_G\times F_X} \ar[r]
                & X \ar[d]^{F_X}  \\
  G\times X \ar[r]
                & X             .}$$
Then, the morphism $F_{X}: X\rightarrow X$ naturally induces a functor $$F_{X}^*: \mathcal{D}_{G,w}^b(X)\rightarrow \mathcal{D}_{G,w}^b(X).$$
Let $x\in X^F$ be a closed point. For any $(\mathcal{F},j)$ in $\mathcal{D}_{G,w}^b(X)$, we get automorphisms $$F_{i,x}:\mathcal{H}_G^i(\mathcal{F})_{\mid x}\rightarrow\mathcal{H}_G^i(\mathcal{F})_{\mid x}.$$
In the same way, one can define the $G$-equivariant version of $\chi^{F^n}$ for $n\in \mathbb{N}$ by:
$$\chi^F_{\mathcal{F}}(x)=\sum_{i}(-1)^itr(F_{i,x},\mathcal{H}_G^i(\mathcal{F})_{\mid x})$$
and
$$\chi^{F^n}_{\mathcal{F}}(x)=\sum_{i}(-1)^itr(F^n_{i,x},\mathcal{H}_G^i(\mathcal{F})_{\mid x}).$$
In particular, $\chi^{F^1}=\chi^F.$

\begin{lemma}
For any $(\mathcal{F},j)\in\mathcal{D}_{G,w}^b(X)$, $\chi_{\mathcal{F}}^{F^n}$ is a $G$-equivariant function.
\end{lemma}
\begin{proof}
For any $x$ and $y$ in the same $G$-orbit of $X$, there exists an element $g\in G$ such that $g.x=y$.
Consider the following commutative diagram
$$\xymatrix{
  \mathcal{H}^i_{G}(\mathcal{F})_{\mid x} \ar[d]^{g^\ast} \ar[r]^{F_x}
                &\mathcal{H}^i_{G}(\mathcal{F})_{\mid x} \ar[d]^{g^\ast} \\
  \mathcal{H}^i_{G}(g^{\ast}\mathcal{F})_{\mid y} \ar[r]^{F_y}
                & \mathcal{H}^i_{G}(g^{\ast}\mathcal{F})_{\mid y}.}$$
Since $g^{\ast}\mathcal{F}\simeq\mathcal{F}$, we have
$$\xymatrix{
  \mathcal{H}^i_{G}(\mathcal{F})_{\mid x} \ar[d]^{g^\ast} \ar[r]^{F_x}
                &\mathcal{H}^i_{G}(\mathcal{F})_{\mid x} \ar[d]^{g^\ast} \\
  \mathcal{H}^i_{G}(\mathcal{F})_{\mid y} \ar[r]^{F_y}
                & \mathcal{H}^i_{G}(\mathcal{F})_{\mid y}.}$$
By the definition of $\chi^{F^n}$, $\chi^{F^n}_{\mathcal{F}}(x)=\chi^{F^n}_{\mathcal{F}}(y)$. That is $\chi_{\mathcal{F}}^{F^n}$ is a $G$-equivariant function.
\end{proof}

In the $G$-equivariant case, we also have Theorem \ref{tool}. Hence the function $\chi^{F^n}_{\mathcal{F}}$ also only depends on the isomorphism class of $\mathcal{F}$ in $\mathcal{D}_{G,w}^b(X)$. Hence, we obtain a map $\chi^{F^n}: K_{G,w}(X)\rightarrow \mathcal{CF}_G(X^{F^n}).$

Let $Q$ be a finite quiver.
Given a dimension vector $\alpha=\sum_{i\in Q_0}\alpha_ii\in\mathbb{N} Q_0,$ the variety
$\mathbb{E}_{\alpha}$ and the algebraic group
$G_{\alpha}$ are defined in Section 3. Both of them have natural $\mathbb{F}_q$-structures. Consider the following diagram
$$
\xymatrix{\mathbb{E}_{\alpha}\times\mathbb{E}_{\beta}&\mathbb{E}'\ar[r]^{p_2}\ar[l]_--{p_1}&\mathbb{E}^{''}\ar[r]^{p_3}&\mathbb{E}_{\alpha+\beta}}.
$$
This induces a functor $$\mathbf{m}_{\alpha,\beta}: \mathcal{D}^b_{G_{\alpha}\times G_{\beta},w}(\mathbb{E}_{\alpha}\times\mathbb{E}_{\beta}) \rightarrow\mathcal{D}^b_{G_{\alpha+\beta},w}(\mathbb{E}_{\alpha+\beta})$$ described as the composition of the following functors:
$$
\xymatrix @-0.9pc{\mathcal{D}^b_{G_{\alpha}\times G_{\beta},w}(\mathbb{E}_{\alpha}\times\mathbb{E}_{\beta})\ar[r]^--{\mathfrak{p}^*_1}&\mathcal{D}^b_{G_{\alpha}\times G_{\beta}\times G_{\alpha+\beta},w}(\mathbb{E}')\ar[r]^-{(\mathfrak{p}_2)_{b}}&\mathcal{D}^b_{G_{\alpha+\beta},w}(\mathbb{E}^{''})\ar[r]^{(\mathfrak{p}_3)_{!}}&\mathcal{D}^b_{G_{\alpha+\beta},w}(\mathbb{E}_{\alpha+\beta})},
$$
where $(\mathfrak{p}_2)_b$ is the inverse of the pull-back functor $$\mathfrak{p}_2^*:\mathcal{D}^b_{G_{\alpha+\beta},w}(\mathbb{E}^{''})\rightarrow \mathcal{D}^b_{G_{\alpha}\times G_{\beta}\times G_{\alpha+\beta},w}(\mathbb{E}'),$$ which is an equivalence of derived categories.
By definition, $$\chi^F_{(\mathfrak{p}_2)_b(\mathcal{K})}=\frac{1}{|G_{\alpha}\times G_{\beta}|}\chi^F_{(\mathfrak{p}_2)_{!}(\mathcal{K})}$$
for $\mathcal{K}\in \mathcal{D}^b_{G_{\alpha+\beta},m}(\mathbb{E}^{''})$ since $p_2$ is a principal $G_{\alpha}\times G_{\beta}$-bundle.

Applying Theorem \ref{tool}, we obtain the following commutative diagrams
$$
\xymatrix @C-0.9pc{\mathcal{D}^b_{G_{\alpha}\times G_{\beta},w}(\mathbb{E}_{\alpha}\times\mathbb{E}_{\beta})\ar[d]\ar[r]^--{\mathfrak{p}^*_1}&\mathcal{D}^b_{G_{\alpha}\times G_{\beta}\times G_{\alpha+\beta},w}(\mathbb{E}')\ar[d]\ar[r]^-{(\mathfrak{p}_2)_{b}}&\mathcal{D}^b_{G_{\alpha+\beta},w}(\mathbb{E}^{''})\ar[d]\ar[r]^{(\mathfrak{p}_3)_{!}}&\mathcal{D}^b_{G_{\alpha+\beta},w}(\mathbb{E}_{\alpha+\beta})\ar[d]\\
\mathcal{CF}_{G_{\alpha}\times G_{\beta}}(\mathbb{E}^F_{\alpha}\times\mathbb{E}^F_{\beta})\ar[r]^{p^*_1}&\mathcal{CF}_{G_{\alpha}\times G_{\beta}\times G_{\alpha+\beta}}(\mathbb{E}'^F)\ar[r]^--{\tau}&\mathcal{CF}_{G_{\alpha+\beta}}(\mathbb{E}''^{F})\ar[r]^{(p_3)_{!}}&\mathcal{CF}_{G_{\alpha+\beta}}(\mathbb{E}^F_{\alpha+\beta})}
$$
where $\tau={\frac{1}{|G_{\alpha}\times G_{\beta}|}(p_2)_{!}}$.
Hence, the linear functor $$\mathbf{m}_{\alpha,\beta}: \mathcal{D}^b_{G_{\alpha}\times G_{\beta},w}(\mathbb{E}_{\alpha}\times\mathbb{E}_{\beta})\rightarrow \mathcal{D}^b_{G_{\alpha+\beta},w}(\mathbb{E}_{\alpha+\beta})$$ satisfies that the following diagram is commutative
\begin{equation}\label{commutative-diagran-mult}
\xymatrix{\mathcal{D}^b_{G_{\alpha}\times G_{\beta},w}(\mathbb{E}_{\alpha}\times\mathbb{E}_{\beta})\ar[d]^{\chi^F}\ar[r]^-{\mathbf{m}_{\alpha,\beta}}& \mathcal{D}^b_{G_{\alpha+\beta},w}(\mathbb{E}_{\alpha+\beta})\ar[d]^{\chi^F}\\
\mathcal{CF}_{G^F_{\alpha}\times G^F_{\beta}}(\mathbb{E}^F_{\alpha}\times\mathbb{E}^F_{\beta})\ar[r]^-{\underline{m}_{\alpha,\beta}}&\mathcal{CF}_{G^F_{\alpha+\beta}}(\mathbb{E}^F_{\alpha+\beta}).}
\end{equation}

\begin{lemma}\label{lemma_1}
For simple perverse sheaves $\mathcal{L}\in\mathcal{D}^b_{G_{\alpha}\times G_{\beta},w}(\mathbb{E}_{\alpha}\times\mathbb{E}_{\beta})$, $\mathbf{m}_{\alpha,\beta}(\mathcal{L})$ is still semisimple in $\mathcal{D}^b_{G_{\alpha+\beta},w}(\mathbb{E}_{\alpha+\beta})$.
\end{lemma}
\begin{proof}
Since $\mathfrak{p}_1$ is smooth with connect fibres, $\mathfrak{p}^*_1(\mathcal{L})$ is still semisimple by Section 4.2.4 and 4.2.5 in \cite{BBD}.
Since $(\mathfrak{p}_2)_b$ is a equivalence of categories, $(\mathfrak{p}_2)_{b}\mathfrak{p}^*_1(\mathcal{L})$ is still semisimple.
At last, the fact that $\mathfrak{p}_3$ is proper implies that $\mathbf{m}_{\alpha,\beta}(\mathcal{L})=(\mathfrak{p}_3)_!(\mathfrak{p}_2)_{b}\mathfrak{p}^*_1(\mathcal{L})$ is also semisimple.
\end{proof}

By Lemma \ref{lemma_1}, the linear functor $\mathbf{m}_{\alpha,\beta}: \mathcal{D}^b_{G_{\alpha}\times G_{\beta},w}(\mathbb{E}_{\alpha}\times\mathbb{E}_{\beta})\rightarrow \mathcal{D}^b_{G_{\alpha+\beta},w}(\mathbb{E}_{\alpha+\beta})$
induces an additive map
$$\mathfrak{m}_{\alpha,\beta}: K_{G_{\alpha}\times G_{\beta},w}(\mathbb{E}_{\alpha}\times\mathbb{E}_{\beta})\rightarrow K_{G_{\alpha+\beta}}(\mathbb{E}_{\alpha+\beta})$$
such that the following diagram is commutative
\begin{equation}\label{commutative-diagran-mult1}
\xymatrix{K_{G_{\alpha}\times G_{\beta},w}(\mathbb{E}_{\alpha}\times\mathbb{E}_{\beta})\ar[d]^{\chi^F}\ar[r]^-{\mathfrak{m}_{\alpha,\beta}}& K_{G_{\alpha+\beta},w}(\mathbb{E}_{\alpha+\beta})\ar[d]^{\chi^F}\\
\mathcal{CF}_{G^F_{\alpha}\times G^F_{\beta}}(\mathbb{E}^F_{\alpha}\times\mathbb{E}^F_{\beta})\ar[r]^-{\underline{m}_{\alpha,\beta}}&\mathcal{CF}_{G^F_{\alpha+\beta}}(\mathbb{E}^F_{\alpha+\beta}).}
\end{equation}

Set $\mathbf{K}_{w}=\bigoplus_{\alpha}K_{G_{\alpha},w}(\mathbb{E}_{\alpha})$ and $\mathcal{CF}^F(Q)=\bigoplus_{\alpha}\mathcal{CF}_{ G_{\alpha}}(\mathbb{E}^F_{\alpha}).$ There is a linear map from $\mathbf{K}_{w}$ to $\mathcal{CF}^F(Q)$ induced by $\chi^F$. For simplicity, we also denote it by $\chi^F.$
For $\mathcal{M}\in \mathcal{D}_{G_{\alpha},w}^b(\mathbb{E}_{\alpha})$ and $\mathcal{N}\in \mathcal{D}_{G_{\beta},w}^b(\mathbb{E}_{\beta})$, define $[\mathcal{M}]*[\mathcal{N}]:=[\mathbf{m}_{\alpha,\beta}(\mathcal{M}\boxtimes \mathcal{N})].$
Then the linear maps $\mathfrak{m}$ and $\underline{m}$ endow $\mathbf{K}_{w}$ and $\mathcal{CF}^F(Q)$ with multiplication structures, respectively.
Using Diagram (\ref{commutative-diagran-mult1}), we obtain the following result.

\begin{Prop}\label{algebrahomo}
The $\mathbb{Z}$-linear map $\chi^F: \mathbf{K}_{w}\rightarrow \mathcal{CF}^F(Q)$ is a ring homomorphism.
\end{Prop}


Consider the diagram
$$
\xymatrix{\mathbb{E}_{\alpha}\times\mathbb{E}_{\beta}&F_{\alpha,\beta}\ar[l]_-{\kappa}\ar[r]^{i}&\mathbb{E}_{\alpha+\beta}}.
$$
This induces a functor $$\tilde{\mathbf{\Delta}}_{\alpha, \beta}: \mathcal{D}^b_{G_{\alpha+\beta},w}(\mathbb{E}_{\alpha+\beta})\rightarrow \mathcal{D}^b_{G_{\alpha}\times G_{\beta},w}(\mathbb{E}_{\alpha}\times \mathbb{E}_{\beta})$$ as the composition of functors:
$$
\xymatrix{\mathcal{D}^b_{G_{\alpha}\times G_{\beta},w}(\mathbb{E}_{\alpha}\times \mathbb{E}_{\beta})&
\mathcal{D}^b_{G_{\alpha+\beta},w}(F_{\alpha, \beta})
\ar[l]_-{\kappa_{!}}&\mathcal{D}^b_{G_{\alpha+\beta},w}(\mathbb{E}_{\alpha+\beta})\ar[l]_-{i^*}}.
$$
Applying Theorem \ref{tool}, we have the commutative diagram
$$
\xymatrix{\mathcal{D}^b_{G_{\alpha+\beta},w}(\mathbb{E}_{\alpha+\beta})\ar[d]^{\chi^F}\ar[r]^-{\tilde{\mathbf{\Delta}}_{\alpha, \beta}}&\mathcal{D}^b_{G_{\alpha}\times G_{\beta},w}(\mathbb{E}_{\alpha}\times\mathbb{E}_{\beta})\ar[d]^{\chi^F} \\
\mathcal{CF}_{G^F_{\alpha+\beta}}(\mathbb{E}^F_{\alpha+\beta})\ar[r]^-{\tilde{\delta}_{\alpha, \beta}}&\mathcal{CF}_{G^F_{\alpha}\times G^F_{\beta}}(\mathbb{E}^F_{\alpha}\times\mathbb{E}^F_{\beta}).}
$$

\begin{lemma}\cite{Braden}\footnote{The authors thank  Hiraku Nakajima  for pointing out Reference \cite{Braden}.}\label{lemma_2}
For simple perverse sheaves $\mathcal{L}\in\mathcal{D}^b_{G_{\alpha+\beta},w}(\mathbb{E}_{\alpha+\beta})$, $\tilde{\mathbf{\Delta}}_{\alpha,\beta}(\mathcal{L})$ is still semisimple in $\mathcal{D}^b_{G_{\alpha}\times G_{\beta},w}(\mathbb{E}_{\alpha}\times\mathbb{E}_{\beta})$.
\end{lemma}

\begin{proof}
Since $\tilde{\mathbf{\Delta}}_{\alpha,\beta}$ is a hyperbolic localization (\cite{Braden}).
\end{proof}

By Lemma \ref{lemma_2}, the linear functor $\tilde{\mathbf{\Delta}}_{\alpha,\beta}:\mathcal{D}^b_{G_{\alpha+\beta},w}(\mathbb{E}_{\alpha+\beta})\rightarrow\mathcal{D}^b_{G_{\alpha}\times G_{\beta},w}(\mathbb{E}_{\alpha}\times\mathbb{E}_{\beta})$
induces a linear map
$$\tilde{\Delta}_{\alpha,\beta}:K_{G_{\alpha+\beta},w}(\mathbb{E}_{\alpha+\beta})\rightarrow K_{G_{\alpha}\times G_{\beta},w}(\mathbb{E}_{\alpha}\times\mathbb{E}_{\beta})$$
such that the following diagram is commutative
$$
\xymatrix{K_{G_{\alpha+\beta},w}(\mathbb{E}_{\alpha+\beta})\ar[d]^{\chi^F}\ar[r]^-{\Delta_{\alpha, \beta}}&K_{G_{\alpha}\times G_{\beta},w}(\mathbb{E}_{\alpha}\times\mathbb{E}_{\beta})\ar[d]^{\chi^F} \\
\mathcal{CF}_{G_{\alpha+\beta}}(\mathbb{E}^F_{\alpha+\beta})\ar[r]^-{\delta_{\alpha, \beta}}&\mathcal{CF}_{G_{\alpha}\times G_{\beta}}(\mathbb{E}^F_{\alpha}\times\mathbb{E}^F_{\beta}),}
$$
where $\Delta_{\alpha, \beta}$ is induced by $\mathbf{\Delta}_{\alpha, \beta}=\kappa_{!}i^*[(-2)\sum_{i\in Q_0}\alpha_i\beta_i](-\sum_{i\in Q_0}\alpha_i\beta_i).$

For any $[\mathcal{L}]\in \mathbf{K}_{w}$ such that $\mathcal{L}\in\mathcal{D}^b_{G_{\gamma},w}(\mathbb{E}_{\gamma})$, define $$\Delta([\mathcal{L}])=\sum_{\alpha,\beta;\alpha+\beta=\gamma}\Delta_{\alpha,\beta}([\mathcal{L}]).$$
In the same way as Proposition \ref{algebrahomo}, we obtain the following result.

\begin{Prop}\label{coalghomomorphism}
The $\mathbb{Z}$-linear map $\chi^F: \mathbf{K}_{w}\rightarrow \mathcal{CF}^F(Q)$ satisfies $\delta\circ \chi^F=\chi^F\circ \Delta$.
\end{Prop}

In Section 3, we have shown that there exists a comultiplication structure over $\mathcal{CF}^F(Q)$. By Green's theorem, the comultiplication is compatible with the multiplication structure and then $\mathcal{CF}^F(Q)$ is a bialgebra. Naturally, one would like to check whether the map $\Delta$ is compatible with the multiplication structure of $\mathbf{K}_{w}$.
Let $\mathcal{CF}^{F^n}(Q)=\bigoplus_{\alpha}\mathcal{CF}_{G_{\alpha}}(\mathbb{E}^{F^n}_{\alpha}).$
Similarly to $\chi^F:\mathbf{K}_{w}\rightarrow\mathcal{CF}^{F}(Q)$, we have $\chi^{F^n}:\mathbf{K}_{w}\rightarrow\mathcal{CF}^{F^n}(Q)$
for any $n\in\mathbb{N}$.

\begin{theorem}\cite[Theorem 12.1]{KW}\label{injective}
The ring  homomorphism $$\chi=\prod_{n\in \mathbb{N}}\chi^{F^n}: \mathbf{K}_{w}\rightarrow \prod_{n\in \mathbb{N}}\mathcal{CF}^{F^n}(Q)$$ is injective.
\end{theorem}

To simplify the notations, we set $$\mathcal{D}_{\alpha,\beta}=\mathcal{D}^b_{G_{\alpha}\times G_{\beta},w}(\mathbb{E}_{\alpha}\times \mathbb{E}_{\beta}) $$$$\mathcal{D}_{\alpha',\alpha'',\beta',\beta''}=\mathcal{D}^b_{G_{\alpha'}\times G_{\alpha''}\times G_{\beta'}\times G_{\beta''},w}(\mathbb{E}_{\alpha'}\times \mathbb{E}_{\alpha''}\times \mathbb{E}_{\beta'}\times \mathbb{E}_{\beta''})$$ and
$$K_{\alpha,\beta}=K_{G_{\alpha}\times G_{\beta},w}(\mathbb{E}_{\alpha}\times \mathbb{E}_{\beta})$$$$ K_{\alpha',\alpha'',\beta',\beta''}=K_{G_{\alpha'}\times G_{\alpha''}\times G_{\beta'}\times G_{\beta''},w}(\mathbb{E}_{\alpha'}\times \mathbb{E}_{\alpha''}\times \mathbb{E}_{\beta'}\times \mathbb{E}_{\beta''}).$$

One can view $\mathbb{E}_{\alpha}\times \mathbb{E}_{\beta}$ as the variety of representations over $Q\times Q$ with dimension vector $(\alpha, \beta)$ and then $\mathbb{E}_{(\alpha, \beta)}=\mathbb{E}_{(\alpha, \beta)}(Q\times Q)=\mathbb{E}_{\alpha}(Q)\times \mathbb{E}_{\beta}(Q)=\mathbb{E}_{\alpha}\times \mathbb{E}_{\beta}$.
As in Section 3, substituting $Q\times Q$ for $Q$, we obtain the diagrams
$$
\xymatrix{\mathbb{E}_{(\alpha',\alpha'')}\times\mathbb{E}_{(\beta', \beta'')}&\mathbb{E}'(Q\times Q)\ar[r]^{p_2}\ar[l]_--{p_1}&\mathbb{E}''(Q\times Q)\ar[r]^-{p_3}&\mathbb{E}_{(\alpha'+\beta', \alpha''+\beta'')}}
$$
and
$$
\xymatrix{\mathbb{E}_{(\alpha',\beta')}\times\mathbb{E}_{(\alpha'', \beta'')}&F_{(\alpha',\beta'),(\alpha'',\beta'')}(Q\times Q)\ar[l]_-{\kappa}\ar[r]^-{i}&\mathbb{E}_{(\alpha,\beta)}}.
$$
for $\alpha=\alpha'+\alpha''$ and $\beta=\beta'+\beta''.$
There is a natural homeomorphism $\mathbb{E}_{\alpha'}\times \mathbb{E}_{\alpha''}\times \mathbb{E}_{\beta'}\times \mathbb{E}_{\beta''}\rightarrow \mathbb{E}_{\alpha'}\times \mathbb{E}_{\beta'}\times \mathbb{E}_{\alpha''}\times \mathbb{E}_{\beta''}$. It induces an equivalence $\tau: \mathcal{D}_{\alpha',\beta',\alpha'',\beta''}\rightarrow \mathcal{D}_{\alpha',\alpha'',\beta',\beta''}.$ In the same way as above, there exists the induction functor $\mathbf{m}_{(\alpha',\alpha'',\beta',\beta'')}=(\mathfrak{p}_3)_{!}(\mathfrak{p}_2)_b(\mathfrak{p}_1)^*[2\langle\alpha', \beta''\rangle](\langle\alpha', \beta''\rangle): \md_{\alpha', \alpha'', \beta',\beta''}\rightarrow \md_{\alpha'+\beta',\alpha''+\beta''}$ and the restriction functor $\tilde{\mathbf{\Delta}}_{(\alpha',\beta'),(\alpha'',\beta'')}=(\kappa)_{!}(i)^*: \md_{\alpha, \beta}\rightarrow \md_{\alpha',\beta',\alpha'',\beta''}$. The functors induce the additive maps over Grothendieck groups $\mathfrak{m}: K_{\alpha',\alpha'',\beta',\beta''}\rightarrow K_{\alpha'+\beta',\alpha''+\beta''}$ and $\tilde{\Delta}_{\alpha',\alpha'',\beta',\beta''}:  K_{\alpha, \beta}\rightarrow K_{\alpha',\beta',\alpha'',\beta''}$.
In the same way, one can define $$\mathbf{\Delta}_{(\alpha',\beta'),(\alpha'',\beta'')}=\tilde{\mathbf{\Delta}}_{(\alpha',\beta'),(\alpha'',\beta'')}[-2\sum_{i\in Q_0}\alpha'_i\alpha''_i+\beta'_i\beta''_i](-\sum_{i\in Q_0}\alpha'_i\alpha''_i+\beta'_i\beta''_i).$$ For  $\mathcal{L}_1\in \md_{\alpha}=\mathcal{D}^b_{G_{\alpha}, w}(\mathbb{E}_{\alpha})$ and $\mathcal{L}_2\in \md_{\beta}=\mathcal{D}^b_{G_{\beta}, w}(\mathbb{E}_{\beta})$, we set $\mathcal{L}_1*\mathcal{L}_2=\mathbf{m}_{\alpha,\beta}(\mathcal{L}_1\boxtimes\mathcal{L}_2)$. Here, we apply the functor $\md_{\alpha}\times \md_{\beta}\rightarrow \md_{\alpha,\beta}$ sending $(U, V)$ to $U\boxtimes V.$ Similarly, for $\mathcal{F}_1\in \md_{\alpha',\alpha''}$ and $\mathcal{F}_2\in\md_{\beta', \beta''}$, define $\mathcal{F}_1*\mathcal{F}_2=\mathbf{m}_{(\alpha',\alpha''),(\beta',\beta'')}(\mathcal{F}_1\boxtimes\mathcal{F}_2).$ The following theorem can be viewed as the categorification of the Green formula in Section 2.
\begin{theorem}\label{thm_main}
Let $\alpha, \beta, u, v$ be dimension vectors such that $\alpha+\beta=u+v=\gamma$ and $\mathcal{L}_1\in \mathcal{D}^b_{G_{\alpha}, w}(\mathbb{E}_{\alpha})$, $\mathcal{L}_2\in \mathcal{D}^b_{G_{\beta}, w}(\mathbb{E}_{\beta})$ be two simple perverse sheaves. Set $\mathcal{N}=\{\lambda=(\alpha',\alpha'',\beta',\beta'')\mid \alpha=\alpha'+\alpha'',\beta=\beta'+\beta'', u=\alpha'+\beta', v=\alpha''+\beta''\}$. Then we have
$$
\mathbf{\Delta}_{u, v}(\mathcal{L}_1*\mathcal{L}_2)=\bigoplus_{\lambda\in\mathcal{N}}\mathbf{\Delta}_{\alpha',\alpha''}(\mathcal{L}_1)*\mathbf{\Delta}_{\beta', \beta''}(\mathcal{L}_2).
$$
\end{theorem}
\begin{proof}
Let $\gamma=\alpha+\beta$ and $N^{\tau}=\{\lambda=(\alpha',\beta',\alpha'',\beta'')\mid \alpha=\alpha'+\alpha'',\beta=\beta'+\beta'', u=\alpha'+\beta', v=\alpha''+\beta''\}$.  Consider the diagram
$$
\xymatrix{K_{\alpha, \beta}\ar[d]_{(\Delta^{tw}_{\lambda})_{\lambda\in \mathcal{N}^{\tau}}}\ar[rr]^{\mathfrak{m}_{\alpha,\beta}}&&K_{\gamma}\ar[d]^{\Delta^{tw}_{(u,v)}}\\
\coprod_{\lambda\in \mathcal{N}^{\tau}}K_{\lambda}\ar[r]^{(\tau_{\lambda})_{\lambda\in\mathcal{N}^{\tau}}}&\coprod_{\lambda\in \mathcal{N}}K_{\lambda}\ar[r]^-{(\mathfrak{m}_{\lambda})}&K_{u,v}.}
$$

For simple perverse sheaves $\mathcal{L}_1\in \mathcal{D}^b_{G_{\alpha}, w}(\mathbb{E}_{\alpha})$ and $\mathcal{L}_2\in \mathcal{D}^b_{G_{\beta}, w}(\mathbb{E}_{\beta})$, $\mathcal{L}_1*\mathcal{L}_2$ is still semisimple  by Lemma \ref{lemma_1} and then $\mathbf{\Delta}_{u,v}(\mathcal{L}_1*\mathcal{L}_2)$ is semisimple by Lemma \ref{lemma_2}.
The rightside term $\mathbf{\Delta}_{\alpha',\alpha''}(\mathcal{L}_1)*\mathbf{\Delta}_{\beta', \beta''}(\mathcal{L}_2)$ of the equation
is also semisimple by Lemma \ref{lemma_1} and \ref{lemma_2}. It is equivalent to show that the equation holds in $K_{w}.$  Hence, it is enough to prove that the diagram is commutative.  By Theorem \ref{injective}, it is equivalent to show that
$$
\chi^{F^n}\circ\Delta_{u,v}\circ \mathfrak{m}_{\alpha,\beta}=\chi^{F^n}\circ\mathfrak{m}_{\lambda}\circ(\tau_{\lambda})_{\lambda\in\mathcal{N}^{\tau}}\circ(\Delta_{\lambda})_{\lambda\in\mathcal{N}^{\tau}}
$$
for any $n\in \mathbb{N}.$ By Proposition \ref{algebrahomo} and \ref{coalghomomorphism}, one need to prove  $\delta_{u,v}\underline{m}_{\alpha,\beta}=\underline{m}\delta_1.$ By Theorem \ref{green-geometry}, we have the desired equation.
\end{proof}

The ring $\mathbf{K}_{w}$ has a natural $\mathbb{Z}$-module structure with a basis consisting of the isomorphism classes of simple perverse sheaves.
There is also a natural $\mathbb{A}=\mathbb{Z}[v,v^{-1}]$-module structure on $\mathbf{K}_{w}$ by $$v[\mathcal{L}]=[\mathcal{L}[1](\frac{1}{2})] \mbox{ and } v^{-1}[\mathcal{L}]=[\mathcal{L}[-1](-\frac{1}{2})]$$
for any dimension vector $\alpha$ and $\mathcal{L}\in\mathcal{D}^b_{G_{\alpha},w}(\mathbb{E}_{\alpha})$.
Consider the equivalence classes of this basis under Tate twist and choose a representative in any equivalence class. The set of all representatives is denoted by $\mathbf{B}_w$.

\begin{Prop}
The ring $\mathbf{K}_{w}$ is a free $\mathbb{A}$-module with $\mathbf{B}_w$ as a basis.
\end{Prop}

For a variety $X$ as above, we denote by $\md^b_{im, G, w}(X)$ the subcategory of $\md^b_{G,w}(X)$ consisting of $\iota$-mixed Weil complexes of integer weights. Let $K_{im,G,w}(X)$ be the corresponding  Grothendieck group.

Set $$\mathbf{I}_w=\bigoplus_{\alpha} K_{im, G_{\alpha}, w}(\mathbb{E}_{\alpha}).$$
For any $\alpha\in\mathbb{N}Q_0$, let $\mathcal{I}_{w,\alpha}$ be the set of direct sums of $\iota$-pure simple perverse sheaves in $\md^b_{im, G, w}(\mathbb{E}_{\alpha})$ with weight $0$, $$\mathbf{B}_\alpha=\{[\mathcal{L}]\,\,|\,\,\mathcal{L}\in\mathcal{I}_{w,\alpha} \textrm{ is simple}\}$$ and $\mathbf{B}=\sqcup_{\alpha}\mathbf{B}_\alpha$.

\begin{Prop}
As the $\mathbb{A}$-submodule of $\mathbf{K}_{w}$, $\mathbf{I}_{w}$ is free and has $\mathbf{B}$ as a basis.
\end{Prop}

The basis $\mathbf{B}$ is called the canonical basis of $\mathbf{I}_{w}$.

\begin{theorem}\label{sub_I_K}
The $\mathbb{A}$-module $\mathbf{I}_{w}$ is a subring of $\mathbf{K}_{w}$ such that $\Delta_{im}\circ \mathfrak{m}=\mathfrak{m}\circ\Delta_{im}$ where $\Delta_{im}$ is the natural restriction of $\Delta$.
\end{theorem}
\begin{proof}
Due to \cite[I.9, II.12]{KW}, we have the following commutative diagrams
$$
\xymatrix{\mathcal{D}^b_{im, G_{\alpha}\times G_{\beta}, w}(\mathbb{E}_{\alpha}\times \mathbb{E}_{\beta})\ar[d]\ar[r]^-{\mathbf{m}_{im,\alpha, \beta}}&\mathcal{D}_{im, G_{\alpha+\beta}, w}(\mathbb{E}_{\alpha+\beta})\ar[d] \\
\mathcal{D}^b_{G_{\alpha}\times G_{\beta}, w}(\mathbb{E}_{\alpha}\times \mathbb{E}_{\beta})\ar[r]^-{\mathbf{m}_{\alpha, \beta}}&\mathcal{D}_{G_{\alpha+\beta}, w}(\mathbb{E}_{\alpha+\beta})}
$$
and
$$
\xymatrix{\mathcal{D}_{im, G_{\alpha+\beta}, w}(\mathbb{E}_{\alpha+\beta})\ar[d]\ar[r]^-{\mathbf{\Delta}_{im,\alpha, \beta}}&\mathcal{D}^b_{im, G_{\alpha}\times G_{\beta}, w}(\mathbb{E}_{\alpha}\times \mathbb{E}_{\beta})\ar[d] \\
\mathcal{D}_{G_{\alpha+\beta}, w}(\mathbb{E}_{\alpha+\beta})\ar[r]^-{\mathbf{\Delta}_{\alpha, \beta}}&\mathcal{D}^b_{G_{\alpha}\times G_{\beta}, w}(\mathbb{E}_{\alpha}\times \mathbb{E}_{\beta}),}
$$
where $\mathbf{m}_{im,\alpha, \beta}$ and $\mathbf{\Delta}_{im,\alpha, \beta}$ are the natural restriction of $\mathbf{m}_{\alpha, \beta}$ and $\mathbf{\Delta}_{\alpha, \beta}$, respectively.
Hence the $\mathbb{A}$-module $\mathbf{I}_{w}$ is a subring of $\mathbf{K}_{w}$ such that $\Delta_{im}\circ \mathfrak{m}=\mathfrak{m}\circ\Delta_{im}$.
\end{proof}

Set ${_{\mathbb{Q}(v)}\mathbf{K}}_{w}=\mathbf{K}_{w}\otimes_{\mathbb{Z}[v, v^{-1}]}\mathbb{Q}(v)$ and ${_{\mathbb{Q}(v)}\mathbf{I}}_{w}=\mathbf{I}_{w}\otimes_{\mathbb{Z}[v, v^{-1}]}\mathbb{Q}(v)$

\begin{theorem}\label{closed_multi}
${_{\mathbb{Q}(v)}\mathbf{I}}_{w}$ is a  subalgebra of ${_{\mathbb{Q}(v)}\mathbf{K}}_{w}$ such that $\Delta_{im}\circ \mathfrak{m}=\mathfrak{m}\circ\Delta_{im}$.
\end{theorem}

We also consider the twisted version of $\mathbf{K}_{w}$ in order to preserve the subcategories of perverse sheaves by defining
$$
\mathfrak{m}_{\alpha,\beta}^t=\mathfrak{m}_{\alpha,\beta}[\{\alpha, \beta\}](\frac{\{\alpha, \beta\}}{2}),
$$
where $\{\alpha, \beta\}=\sum_{i\in Q_0}\alpha_i\beta_i+\sum_{h\in Q_1}\alpha_{s(h)}\beta_{t(h)}$
and
$$
\Delta_{\alpha, \beta}^t=\Delta_{\alpha,\beta}[-\langle\alpha, \beta\rangle](-\frac{\langle\alpha,\beta\rangle}{2}).
$$
We denote by $\mathbf{K}_{w}^{tw}$ the twist of $\mathbf{K}_{w}$ with the multiplication and comultiplication induced by $\mathfrak{m}_{\alpha,\beta}^t$ and $\Delta_{\alpha, \beta}^t$, respectively.
For $\mathcal{M}\in \mathcal{D}_{G_{\alpha},w}^b(\mathbb{E}_{\alpha})$ and $\mathcal{N}\in \mathcal{D}_{G_{\beta},w}^b(\mathbb{E}_{\beta})$, define $[\mathcal{M}]\cdot[\mathcal{N}]:=\mathfrak{m}^t([\mathcal{M}\boxtimes \mathcal{N}]).$
Similarly, denote by $\mathbf{I}_{w}^{tw}$ the twist of $\mathbf{I}_{w}$.

The following lemma is a simple generalization of Lusztig's construction over quantum groups (\cite[Theorem 3.24]{Schiffmann2}).
\begin{lemma}\label{bialgebra}
There is a ring homomorphism $\chi^{F, tw}: \mathbf{K}_{w}^{tw}\rightarrow \mathcal{CF}^{F, tw}(Q)$ by sending $[\mathcal{L}]$ for $\mathcal{L}\in \mathcal{D}^b_{G_{\alpha},w}(\mathbb{E}_{\alpha})$ to $v^{\dim G_{\alpha}}\chi^F(\mathcal{L})$ such that $\delta^{t}\circ \chi^{F, tw}=\chi^{F, tw}\circ\Delta^t.$
\end{lemma}
\begin{Cor}The ring  homomorphism
$$\chi^{tw}=\prod_{n\in \mathbb{N}}\chi^{F^n, tw}: \mathbf{K}_{w}^{tw}\rightarrow \prod_{n\in \mathbb{N}}\mathcal{CF}^{F^n, tw}(Q)$$ is an injective homomorphism satisfying $\delta^{t}\circ \chi^{F^n, tw}=\chi^{F^n, tw}\circ\Delta^t$ for $n\in \mathbb{N}.$
\end{Cor}

We will endow $\mathbf{K}^{tw}_{w}$ with the structure of the antipode map as an analogue of the antipode over a Ringel-Hall algebra.
Given dimension vectors $\alpha_1,\alpha_2,\alpha_3$, consider the following diagram
$$
\xymatrix{\mathbb{E}_{\alpha_2}\times\mathbb{E}_{\alpha_3}&\mathbb{E}'\ar[r]^{p_2}\ar[l]_--{p_1}&\mathbb{E}^{''}\ar[r]^{p_3}&\mathbb{E}_{\alpha_2+\alpha_3}}.
$$
This induces the following diagram
$$
\xymatrix{\mathbb{E}_{\alpha_1}\times\mathbb{E}_{\alpha_2}\times \mathbb{E}_{\alpha_3}&\mathbb{E}_{\alpha_1}\times\mathbb{E}'\ar[r]^{(1, p_2)}\ar[l]_--{(1, p_1)}&\mathbb{E}_{\alpha_1}\times\mathbb{E}^{''}\ar[r]^-{(1, p_3)}&\mathbb{E}_{\alpha_1}\times\mathbb{E}_{\alpha_2+\alpha_3}}.
$$
Then we obtain a functor $$1\hat{\otimes}\mathbf{m}_{\alpha_2,\alpha_3}: \mathcal{D}^b_{G_{\alpha_1}\times G_{\alpha_2}\times G_{\alpha_3},w}(\mathbb{E}_{\alpha_1}\times\mathbb{E}_{\alpha_2}\times \mathbb{E}_{\alpha_3}) \rightarrow\mathcal{D}^b_{G_{\alpha_1}\times G_{\alpha_2+\alpha_3},w}(\mathbb{E}_{\alpha_1}\times\mathbb{E}_{\alpha_2+\alpha_3}).$$
Similarly, one can define the functor $\mathbf{m}_{\alpha_2,\alpha_3}\hat{\otimes}1.$ Now, we can define the $r$-fold multiplication inductively by setting
$$\mathbf{m}^2_{\alpha_2,\alpha_3}=\mathbf{m}_{\alpha_2,\alpha_3}, \quad \mathbf{m}^r_{\alpha_1,\cdots,\alpha_r}=\mathbf{m}_{\alpha_1, \alpha_2+\cdots+\alpha_r}\circ(1\hat{\otimes}\mathbf{m}^{r-1}_{\alpha_2,\cdots,\alpha_r}).$$
for $r>2.$

In the same way, we inductively define the $r$-fold comultiplication. Consider the following diagram
$$
\xymatrix{\mathbb{E}_{\alpha_1}\times\mathbb{E}_{\alpha_2}&F_{\alpha_1,\alpha_2}\ar[l]_-{\kappa}\ar[r]^{i}&\mathbb{E}_{\alpha_1+\alpha_2}}.
$$
This induces the diagram
$$
\xymatrix{\mathbb{E}_{\alpha_1}\times\mathbb{E}_{\alpha_2}\times\mathbb{E}_{\alpha_3}&F_{\alpha_1,\alpha_2}\times\mathbb{E}_{\alpha_3}\ar[l]_-{(\kappa, 1)}\ar[r]^{(i,1)}&\mathbb{E}_{\alpha_1+\alpha_2}\times \mathbb{E}_{\alpha_3}}.
$$
Then we obtain the functor
$$
\mathbf{\Delta}_{\alpha_1,\alpha_2}\hat{\otimes}1:\mathcal{D}^b_{G_{\alpha_1+\alpha_2}\times G_{\alpha_3},w}(\mathbb{E}_{\alpha_1+\alpha_2}\times\mathbb{E}_{\alpha_3}) \rightarrow\mathcal{D}^b_{G_{\alpha_1}\times G_{\alpha_2}\times G_{\alpha_3},w}(\mathbb{E}_{\alpha_1}\times\mathbb{E}_{\alpha_2}\times \mathbb{E}_{\alpha_3}) .
$$
and then the functor $1\hat{\otimes}\mathbf{\Delta}_{\alpha_1,\alpha_2}$ similarly. The $r$-fold comultiplication can be defined inductively: $$\mathbf{\Delta}^2_{\alpha_1,\alpha_2}=\mathbf{\Delta}_{\alpha_1,\alpha_2}, \quad \mathbf{\Delta}^{r}_{\alpha_1,\cdots,\alpha_r}=(\mathbf{\Delta}^{r-1}_{\alpha_1,\cdots,\alpha_{r-1}}\hat{\otimes}1)\circ \mathbf{\Delta}_{\alpha_1+\cdots+\alpha_{r-1},\alpha_r}$$
for $r>2.$
Explicitly, we have the following functors:
$$
\mathbf{m}^r_{\alpha_1,\cdots,\alpha_r}:\mathcal{D}^b_{\prod_{i=1}^rG_{\alpha_i}, w}(\prod_{i=1}^r\mathbb{E}_{\alpha_i})\rightarrow \mathcal{D}^{b}_{G_{\alpha_1+\cdots+\alpha_r},w}(\mathbb{E}_{\alpha_1+\cdots+\alpha_r})
$$
and
$$
\mathbf{\Delta}^r_{\alpha_1,\cdots,\alpha_r}:\mathcal{D}^{b}_{G_{\alpha_1+\cdots+\alpha_r},w}(\mathbb{E}_{\alpha_1+\cdots+\alpha_r})\rightarrow \mathcal{D}^b_{\prod_{i=1}^rG_{\alpha_i}, w}(\prod_{i=1}^r\mathbb{E}_{\alpha_i}).
$$


Let $\mathbf{m}^{t,r}$ and $\mathbf{\Delta}^{t,r}$ be the twist versions of $\mathbf{m}^{r}$ and $\mathbf{\Delta}^{r}$, respectively. Now we define the functor $\mathbf{S}:\mathcal{D}^b_{G_{\alpha},w}(\mathbb{E}_{\alpha})\rightarrow \mathcal{D}^b_{G_{\alpha},w}(\mathbb{E}_{\alpha})$ by setting
$$
\mathbf{S}(\mathcal{L})=\bigoplus_{r\geq 1}\bigoplus_{\alpha_1,\cdots,\alpha_r\neq 0}\mathbf{m}^{t,r}_{\alpha_1,\cdots,\alpha_r}\circ\mathbf{\Delta}^{t,r}_{\alpha_1,\cdots,\alpha_r}(\mathcal{L})[r] \mbox{ for $[\mathcal{L}]\neq[\mathbf{1}_0]$ and } \mathbf{S}(\mathbf{1}_0)=\mathbf{1}_0.
$$
where $\mathbf{1}_0$ is the constant sheaf on $\mathbb{E}_0$.
Denote by $S$ the induced map over  $K_{G_{\alpha}, w}(\mathbb{E}_{\alpha})$, even $\mathbf{K}^{tw}_{w}$.
\begin{lemma}
The map $\chi^{F,tw}$ satisfies that $\chi^{F, tw}(S([\mathcal{L}]))=\sigma^t(\chi^{F,tw}([\mathcal{L}]))$ for $\mathcal{L}\in \mathcal{D}^b_{G_{\alpha},w}(\mathbb{E}_{\alpha}).$
\end{lemma}
\begin{proof}
This follows from $\chi^{F, tw}\mathfrak{m}^{t, r}=\underline{m}^{t,r}\chi^{F,tw}$ and $\chi^{F, tw}\Delta^{t, r}=\delta^{t,r}\chi^{F,tw}$.
\end{proof}
Hence, we call $S$ the antipode over $\mathbf{K}^{tw}_{w}$.

Define the functor
$\mathbf{S}\hat{\otimes}1:\mathcal{D}^b_{G_{\alpha},w}(\mathbb{E}_{\alpha}\times\mathbb{E}_{\beta})\rightarrow \mathcal{D}^b_{G_{\alpha},w}(\mathbb{E}_{\alpha}\times\mathbb{E}_{\beta})$ by setting
$$
\mathbf{S}\hat{\otimes}1(\mathcal{L})=\bigoplus_{r\geq 1}\bigoplus_{\alpha_1,\cdots,\alpha_r\neq 0}\mathbf{m}^{t,r}_{(\alpha_1,\beta),\cdots,(\alpha_r,0)}\circ\mathbf{\Delta}^{t,r}_{(\alpha_1,\beta),\cdots,(\alpha_r,0)}(\mathcal{L})[r]
$$
for $[\mathcal{L}]\neq[\mathbf{1}_0\otimes\mathcal{L}']$ and $\mathbf{S}\hat{\otimes}1(\mathbf{1}_0\otimes\mathcal{L}')=\mathbf{1}_0\otimes\mathcal{L}'$.
Similarly, we can define the functor
$$1\hat{\otimes}\mathbf{S}:\mathcal{D}^b_{G_{\alpha},w}(\mathbb{E}_{\alpha}\times\mathbb{E}_{\beta})\rightarrow \mathcal{D}^b_{G_{\alpha},w}(\mathbb{E}_{\alpha}\times\mathbb{E}_{\beta}).$$
Denote by ${S}\hat{\otimes}1$ and $1\hat{\otimes}{S}$ the induced map over  $K_{G_{\alpha}\times G_{\beta}, w}(\mathbb{E}_{\alpha}\times\mathbb{E}_{\beta})$, respectively.

\begin{Prop}\label{antipode2}
With the above notations, we have
$$
\mathfrak{m}^t(S\hat{\otimes} 1)\Delta^t([\mathcal{L}])=\mathfrak{m}^t(1\hat{\otimes} S)\Delta^t([\mathcal{L}])=\left\{
\begin{array}{c}
\textrm{$0$ \,\,\,\,\,\,if $[\mathcal{L}]\neq[\mathbf{1}_0]$},\\
\textrm{$[\mathbf{1}_0]$ \,\,\,\,\,\,if $[\mathcal{L}]=[\mathbf{1}_0]$},
\end{array}\right.
$$
where $\mathbf{1}_0$ is the constant sheaf on $\mathbb{E}_0$.
\end{Prop}
\begin{proof}
By Theorem \ref{injective}, we only need to prove that
$$
\chi^{F^k, tw}(\mathfrak{m}^t(S\hat{\otimes} 1)\Delta^t([\mathcal{L}]))=\chi^{F^k, tw}(\mathfrak{m}^t(1\hat{\otimes}S)\Delta^t([\mathcal{L}]))=\left\{
\begin{array}{c}
\textrm{$0$ \,\,\,\,\,\,if $[\mathcal{L}]\neq[\mathbf{1}_0]$},\\
\textrm{$[1_0]$ \,\,\,\,\,\,if $[\mathcal{L}]=[\mathbf{1}_0]$}.
\end{array}\right.
$$
Since we have $\chi^{F^k, tw}\mathfrak{m}^{t, r}=\underline{m}^{t,r}\chi^{F^k, tw}$, $\chi^{F^k, tw}\Delta^{t, r}=\delta^{t,r}\chi^{F^k, tw}$ and $\chi^{F, tw}(S\hat{\otimes}1)=(\sigma^t\otimes1)\chi^{F,tw}$, the equation (\ref{antipode_multi_comulti})
implies the desired result.
\end{proof}

As a corollary of Theorem \ref{sub_I_K}, we have the following lemma.
\begin{lemma}\label{lemma:antipode}
The subring $\mathbf{I}^{tw}_{w}$ is closed under the antipode $S$ of $\mathbf{K}^{tw}_{w}$.
\end{lemma}

Define the functor
$\mathbf{S}\hat{\otimes}\mathbf{S}=(1\hat{\otimes}\mathbf{S})\circ (\mathbf{S}\hat{\otimes}1).$ Denote the induced map over the Grothendieck group by $S\hat{\otimes}S.$


As a consequence of
Theorem \ref{sub_I_K}, Proposition \ref{antipode2} and Lemma \ref{lemma:antipode}, we obtain the main theorem in this section.
\begin{theorem}
The subring $\mathbf{I}^{tw}_{w}$ satisfies the following conditions:
\begin{enumerate}
\item $\Delta^t\circ\mathfrak{m}^t=\mathfrak{m}^t\circ\Delta^t,$
\item $S(x\cdot y)=S(y)\cdot S(x)$,  for any $x,y\in\mathbf{I}^{tw}_{w}$,
\item $\Delta^t(S(x))=(S\hat{\otimes} S){\Delta^{t,op}}(x)$,  for any $x\in\mathbf{I}^{tw}_{w}$,
\item $\mathfrak{m}^t(S\hat{\otimes} 1)\Delta^t([\mathcal{L}])=\mathfrak{m}^t(1\hat{\otimes} S)\Delta^t([\mathcal{L}])=\left\{
\begin{array}{c}
\textrm{$0$ \,\,\,\,\,\,if $[\mathcal{L}]\neq[\mathbf{1}_0]$},\\
\textrm{$[\mathbf{1}_0]$ \,\,\,\,\,\,if $[\mathcal{L}]=[\mathbf{1}_0]$},
\end{array}\right.$
\end{enumerate}
where $\Delta^{t,op}$ is the composition of $\Delta^{t}$ with the natural linear isomorphism
from $K_{G_{\alpha}\times G_{\beta}, w}(\mathbb{E}_{\alpha}\times\mathbb{E}_{\beta})$ to $K_{G_{\alpha}\times G_{\beta}, w}(\mathbb{E}_{\beta}\times\mathbb{E}_{\alpha})$.
\end{theorem}


There is the other version of the Grothendieck group
defined in the same way as \cite{Lusztigbook}. Let $\alpha$ be a dimension vector and $\mathcal{R}_{im,\alpha}$ be the additive category of complexes isomorphic to sums of shifts of simple perverse sheaves in $\md_{im, G_{\alpha}, w}(\mathbb{E}_{\alpha})$.  Define $K(\mathcal{R}_{im,\alpha})$ to be the Grothendieck group of $\mathcal{R}_{im,\alpha}$ as an additive category.   It can be viewed as an $\mathbb{A}$-module by setting $v[\mathcal{L}]=[\mathcal{L}[1](\frac{1}{2})]$ and $v^{-1}[\mathcal{L}]=[\mathcal{L}[-1](-\frac{1}{2})]$ for $\mathcal{L}\in \mathcal{R}_{im,\alpha}$. Set $K(\mathcal{R}_{im})=\bigoplus_{\alpha}K(\mathcal{R}_{im,\alpha})$ and ${_{\mathbb{Q}(v)}{K}}(\mathcal{R}_{im})=K(\mathcal{R}_{im})\otimes_{\mathbb{A}}\mathbb{Q}(v)$. Similarly, given two dimension vectors $\alpha$ and $\beta$, $\mathcal{R}_{im, \alpha,\beta}$ is the additive category of complexes isomorphic to sums of shifts of simple perverse sheaves in $\md_{im, G_{\alpha}, w}(\mathbb{E}_{\alpha}\times \mathbb{E}_{\beta}).$
\begin{Prop}
As the $\mathbb{A}$-module, $K(\mathcal{R}_{im})$ is free and has isomorphism classes of $\iota$-mixed simple perverse sheaves of integer weight as the basis.
\end{Prop}

Applying Lemma \ref{lemma_1} and \ref{lemma_2}, we have
$$
\mathbf{m}_{\alpha,\beta}(\mathcal{R}_{im,\alpha}\boxtimes\mathcal{R}_{im,\beta})\subseteq \mathcal{R}_{im,\alpha+\beta} \mbox{ and } \mathbf{\Delta}_{\alpha,\beta}(\mathcal{R}_{im,\alpha+\beta})\subseteq \mathcal{R}_{im,\alpha,\beta},
$$
for two dimension vectors $\alpha, \beta$.
Then $K(\mathcal{R}_{im})$ and ${_{\mathbb{Q}(v)}{K}}(\mathcal{R}_{im})$ can be endowed with the multiplication and comultiplication. Given dimension vectors $\alpha, \beta, u, v$ with $\alpha+\beta=u+v$, applying Green's theorem,
we obtain the commutative diagram
$$
\xymatrix{K(\mathcal{R}_{im,\alpha,\beta})\ar[d]^{\Delta}\ar[r]^-{\mathfrak{m}_{\alpha, \beta}}&K(\mathcal{R}_{im,\alpha+\beta})\ar[d]^-{\Delta_{u, v}}\\
\coprod_{\lambda\in \mathcal{N}}K(\mathcal{R}_{im,\lambda})\ar[r]^-{\mathfrak{m}}&K(\mathcal{R}_{im,u,v})}
$$
where $\lambda$ and $\mathcal{N}$ are defined as in the proof of Theorem \ref{thm_main}.
\begin{Prop}
Let $K(\mathcal{R})$ be the quotient of $K(\mathcal{R}_{im})$ by the relations $[\mathcal{L}[1]]=-[\mathcal{L}]$ for any simple perverse sheaf $\mathcal{L}\in \mathcal{R}_{im}$. Then $K(\mathcal{R})$ can be viewed as an $\mathbb{A}$-module. Write  ${_{\mathbb{Q}(v)}K}(\mathcal{R})=K(\mathcal{R})\otimes_{\mathbb{A}}\mathbb{Q}(v)$.  There is an isomorphism of $Q(v)$-algebras  between ${_{\mathbb{Q}(v)}K}(\mathcal{R})$ and ${_{\mathbb{Q}(v)}\mathbf{I}}_{w}$.
\end{Prop}

\section{Return to quantum groups}
In this section, we will compare the algebra $\mathbf{I}^{tw}_{w}$ with the categorifical construction of the quantum group associated to a quiver without a loop considered by Lusztig (\cite{Lusztig}) and with loops generalized by Bozec (\cite{Bo}). First, we shall recall some notations in \cite{Lusztig} and \cite{Bo}.


Let $Q=(Q_0, Q_1, s, t)$ be a quiver. A vertex $i\in Q_0$ is called imaginary if there is at least one loop at $i.$   Given a dimension vector $\alpha=\sum_{i\in Q_0}\alpha_ii\in\mathbb{N}Q_0,$ define the variety
$$
\mathbb{E}_{\alpha}:=\mathbb{E}_{\alpha}(Q)=\bigoplus_{h\in Q_1}\mathrm{Hom}_{\mathbb{K}}(\mathbb{K}^{\alpha_{s(h)}}, \mathbb{K}^{\alpha_{t(h)}})
$$
with the action of the algebraic group
$G_{\alpha }:=G_{\alpha }(Q)=\prod_{i\in Q_0}GL(\alpha_i,\mathbb{K})$.  Let
$$Y_{\alpha }=\{\mathbf{y}=(\mathbf{i},\mathbf{a})\,\,|\,\,\sum_{l=1}^{k}a_li_l=\alpha \},$$
where $\mathbf{i}=(i_1,i_2,\ldots,i_k),\,\,i_l\in Q_0$, $\mathbf{a}=(a_1,a_2,\ldots,a_k),\,\,a_l\in\mathbb{N}$.
For any element $\mathbf{y}=(\mathbf{i},\mathbf{a})$,
a flag of type $\mathbf{y}$ in $\mathbb{K}^{\alpha }=\bigoplus_{i\in Q_0}\mathbb{K}^{\alpha_i}$ is a sequence
$$\phi=(\mathbb{K}^{\alpha }={V}^k\supset{V}^{k-1}\supset\dots\supset{V}^0=0)$$
of $Q_0$-graded $\mathbb{K}$-vector spaces such that $\underline{\dim}{V}^l/{V}^{l-1}=a_li_l$. Let $F_{\alpha , \mathbf{y}}=F_{\mathbf{y}}$ be the variety of all flags of type $\mathbf{y}$ in $\mathbb{K}^{\alpha }$. For any $x\in \mathbb{E}_{\alpha }$, a flag $\phi$ is called $x$-stable if $x_{h}(V^l_{s(h)})\subset{V}^l_{t(h)}$ for all $l$ and all $h\in H$. Let
$$\tilde{F}_{\alpha , \mathbf{y}}=\tilde{F}_{\mathbf{y}}=\{(x,\phi)\in \mathbb{E}_{\alpha }\times F_{\mathbf{y}}\,\,|\,\,\textrm{$\phi$ is $x$-stable}\}$$
and $\pi_{\alpha , \mathbf{y}}:\tilde{F}_{\mathbf{y}}\rightarrow \mathbb{E}_{\alpha }$
be the projection to $\mathbb{E}_{\alpha }$.

For each $\mathbf{y}\in{Y}_{\alpha }$,  we set $1_{\alpha , \mathbf{y}}=(\pi_{\alpha , \mathbf{y}})_!(1_{\tilde{F}_{\alpha , \mathbf{y}}})$ where $1_{\tilde{F}_{\alpha , \mathbf{y}}}$  is the characteristic function over $\tilde{F}_{\alpha , \mathbf{y}}.$ As in \cite{Lusztig1998}, we denote by $\mathcal{F}_{\alpha }$ the subspace of $\mathcal{CF}_{G_{\alpha }}(\mathbb{E}_{\alpha })$ spanned by $1_{\alpha , \mathbf{y}}$ for $\mathbf{y}\in Y_{\alpha }$.

By the decomposition theorem of Beilinson, Bernstein and Deligne (\cite{BBD}), the Lusztig sheaf $\mathcal{L}_{\alpha , \mathbf{y}}=\mathcal{L}_{\mathbf{y}}=(\pi_{\alpha , \mathbf{y}})_!(\mathbf{1}_{\tilde{F}_{\mathbf{y}}})[d_{\mathbf{y}}](\frac{d_{\mathbf{y}}}{2})\in\mathcal{D}_{G_{\alpha }, w}(\mathbb{E}_{\alpha })$ is a semisimple perverse sheaf, where $\mathbf{1}_{\tilde{F}_{\mathbf{y}}}$ is the constant sheaf over $\tilde{F}_{\mathbf{y}}$ and $d_{\mathbf{y}}=\dim\tilde{F}_{\mathbf{y}}$.

Let $\mathcal{P}_{\alpha}$ be a subcategory of the category of perverse sheaves. The objects in $\mathcal{P}_{\alpha}$ are direct sums of simple perverse sheaves, which are direct summands of $\mathcal{L}_{\mathbf{y}}[r]$ for some $\mathbf{y}\in{Y}_{\alpha }$ and $r\in\mathbb{Z}$. Note that $\mathcal{P}_{\alpha}$ is a subcategory of $\mathcal{I}_{w,\alpha}$.

Let $\mathcal{Q}_{\alpha }$ be the subcategory of $\mathcal{D}_{G_{\alpha }, w}(\mathbb{E}_{\alpha })$, whose objects are the complexes that are isomorphic to finite direct sums of complexes of the form $\mathcal{L}[d](\frac{d}{2})$ for various $\mathcal{L}\in\mathcal{P}_{\alpha}$ and $d\in\mathbb{Z}$. Let $K_{\alpha }$ be the Grothendieck group of $\mathcal{Q}_{\alpha }.$
Define $v[\mathcal{L}]=[\mathcal{L}[1](\frac{1}{2})]$ and $v^{-1}[\mathcal{L}]=[\mathcal{L}[-1](-\frac{1}{2})].$
Then, $K_{\alpha }$ is a free $\mathbb{A}$-module.
Define $$K(\mathcal{Q})=\bigoplus_{\alpha }K_{\alpha }.$$

The functors $\mathbf{m}$ and $\mathbf{\Delta}$ in Section 4 can be restricted to the subcategory $\mathcal{Q}_{\alpha }$. The following observation is given in \cite[Section 9]{Lusztigbook}.
\begin{lemma}
Given two dimension vectors $\alpha, \beta$, we have
$$
\mathbf{m}_{\alpha,\beta}(\mathcal{Q}_{\alpha}\boxtimes\mathcal{Q}_{\beta})\subseteq \mathcal{Q}_{\alpha+\beta} \mbox{ and } \mathbf{\Delta}_{\alpha,\beta}(\mathcal{Q}_{\alpha+\beta})\subseteq \mathcal{Q}_{\alpha}\boxtimes\mathcal{Q}_{\beta}.
$$
\end{lemma}
\begin{proof}
By the definition of the trace map, $\chi^F(\mathbf{1}_{\tilde{F}_{\alpha , \mathbf{y}}})=1_{\tilde{F}_{\alpha , \mathbf{y}}}$.   By Theorem \ref{tool}, we obtain
$$\chi^F(\mathcal{L}_{\alpha , \mathbf{y}})=\chi^F(\pi_{\mathbf{y}})_!(\mathbf{1}_{\tilde{F}_{\mathbf{y}}})[d_{\mathbf{y}}](\frac{d_{\mathbf{y}}}{2}))=(-v)^{d_{\mathbf{y}}}1_{\alpha , \mathbf{y}}.$$
For $\mathcal{L_{\alpha}}, \mathcal{L}_{\beta}$ and $\mathcal{L}_{\alpha+\beta}$ in $\mathcal{Q}_{\alpha}$, $\mathcal{Q}_{\beta}$ and $\mathcal{Q}_{\alpha+\beta}$  respectively,   both $\mathbf{m}_{\alpha,\beta}(\mathcal{L}_{\alpha}\boxtimes\mathcal{L}_{\beta})$ and $\mathbf{\Delta}_{\alpha,\beta}(\mathcal{L}_{\alpha+\beta})$ are semisimple by Lemma  \ref{lemma_1} and \ref{lemma_2}. By Theorem \ref{injective}, it is enough to prove that $\underline{m}_{\alpha,\beta}(1_{\alpha, \mathbf{y}}, 1_{\beta, \mathbf{y}'})\in \mathcal{F}_{\alpha+\beta}$ and $\delta_{\alpha,\beta}(1_{\alpha+\beta, \mathbf{y}''})\in \mathcal{F}_{\alpha}\otimes\mathcal{F}_{\beta}.$  For any element $\mathbf{y}=(\mathbf{i},\mathbf{a})\in Y_{\alpha}$ with $\mathbf{i}=(i_1,i_2,\ldots,i_k),\,\,i_l\in Q_0$, $\mathbf{a}=(a_1,a_2,\ldots,a_k),\,\,a_l\in\mathbb{N}$, we have $1_{\alpha, \mathbf{y}}=(1_{S_{ik}})^{a_k}*(1_{S_{i_{k-1}}})^{a_{k-1}}*\cdots *(1_{S_{i_1}})^{a_1}$ where $1_{S_{i_j}}=1_{\mo^F_{S_{i_j}}}$ for $1\leq j\leq k$. Hence, the lemma follows from Green's theorem and Lemma \ref{coincide}.
\end{proof}

Following \cite[Section 9]{Lusztig}, the embedding of $\mathcal{Q}_{\alpha }$ into $\mathcal{D}^b_{G_{\alpha },w}(\mathbb{E}_{\alpha })$ induces the following diagram:
$$
\xymatrix{\mathcal{Q}_{\alpha}\boxtimes\mathcal{Q}_{\beta}\ar[d]\ar[r]^-{\mathbf{m}_{\alpha, \beta}}&\mathcal{Q}_{\alpha+\beta}\ar[d] \\
\mathcal{D}^b_{G_{\alpha}\times G_{\beta}, w}(\mathbb{E}_{\alpha}\times \mathbb{E}_{\beta})\ar[r]^-{\mathbf{m}_{\alpha, \beta}}&\mathcal{D}_{G_{\alpha+\beta}, w}(\mathbb{E}_{\alpha+\beta})}
$$
and
$$
\xymatrix{\mathcal{Q}_{\alpha+\beta}\ar[d]\ar[r]^-{\mathbf{\Delta}_{\alpha, \beta}}&\mathcal{Q}_{\alpha}\boxtimes\mathcal{Q}_{\beta}\ar[d] \\
\mathcal{D}_{G_{\alpha+\beta}, w}(\mathbb{E}_{\alpha+\beta})\ar[r]^-{\mathbf{\Delta}_{\alpha, \beta}}&\mathcal{D}^b_{G_{\alpha}\times G_{\beta}, w}(\mathbb{E}_{\alpha}\times \mathbb{E}_{\beta}).}
$$
In Section 4, we denote by $\mathcal{R}_{\alpha}=\mathcal{R}_{im,\alpha}$ the additive category of complexes isomorphic to sums of shifts of simple perverse sheaves in $\md_{im, G_{\alpha}, w}(\mathbb{E}_{\alpha})$,
$\mathcal{R}_{\alpha,\beta}=\mathcal{R}_{im, \alpha,\beta}$ the additive category of complexes isomorphic to sums of shifts of simple perverse sheaves in $\md_{im, G_{\alpha}, w}(\mathbb{E}_{\alpha}\times \mathbb{E}_{\beta})$
and
we obtain the commutative diagram
$$
\xymatrix{\mathcal{R}_{\alpha,\beta}\ar[d]^{\mathbf{\Delta}}\ar[r]^-{\mathbf{m}_{\alpha, \beta}}&\mathcal{R}_{\alpha+\beta}\ar[d]^-{\mathbf{\Delta}_{\alpha, \beta}}\\
{\bigoplus}_{\lambda\in \mathcal{N}}\mathcal{R}_{\lambda}\ar[r]^{\mathbf{m}}&\mathcal{R}_{\alpha',\beta'}}
$$
where $\lambda$ and $\mathcal{N}$ are defined as in Section 3. As a natural corollary, the diagram induces the following commutative diagram:
$$
\xymatrix{&\mathcal{Q}_{\alpha}\boxtimes\mathcal{Q}_{\beta}\ar[ld]\ar[dd]^(0.3){\mathbf{\Delta}}\ar[rr]^-{\mathbf{m}_{\alpha, \beta}}&&\mathcal{Q}_{\alpha+\beta}\ar[ld]\ar[dd]^-{\mathbf{\Delta}_{\alpha, \beta}}\\
\mathcal{R}_{\alpha,\beta}\ar[dd]^{\mathbf{\Delta}}\ar[rr]^(.7){\mathbf{m}_{\alpha, \beta}}&&\mathcal{R}_{\alpha+\beta}\ar[dd]^(.3){\mathbf{\Delta}_{\alpha, \beta}}&\\
&\boxtimes_{\lambda\in \mathcal{N}}\mathcal{Q}_{\lambda}\ar[rr]^(.3){\mathbf{m}}\ar[ld]&&\mathcal{Q}_{\alpha'}\boxtimes\mathcal{Q}_{\beta'}\ar[ld]\\
{\bigoplus}_{\lambda\in \mathcal{N}}\mathcal{R}_{\lambda}\ar[rr]^{\mathbf{m}}&&\mathcal{R}_{\alpha',\beta'}&
}.
$$
With the multiplication and comultiplication induced by $\mathfrak{m}_{\alpha, \beta}^t$ and $\Delta_{\alpha,\beta}^t$, respectively, $K(\mathcal{Q})$ can be endowed with the structures of algebra and coalgebra.

\begin{Prop}
$K(\mathcal{Q})$ is the $\mathbb{A}$-submodule of $\mathbf{I}^{tw}_{w}$ with the structure of a bialgebra.
\end{Prop}

In the same way, we can define the antipode over $K(\mathcal{Q})$ as follows. The functors
$$
\mathbf{m}^r_{\alpha_1,\cdots,\alpha_r}:\mathcal{D}^b_{\prod_{i=1}^rG_{\alpha_i}, w}(\prod_{i=1}^r\mathbb{E}_{\alpha_i})\rightarrow \mathcal{D}^{b}_{G_{\alpha_1+\cdots+\alpha_r},w}(\mathbb{E}_{\alpha_1+\cdots+\alpha_r})
$$
and
$$
\mathbf{\Delta}^r_{\alpha_1,\cdots,\alpha_r}:\mathcal{D}^{b}_{G_{\alpha_1+\cdots+\alpha_r},w}(\mathbb{E}_{\alpha_1+\cdots+\alpha_r})\rightarrow \mathcal{D}^b_{\prod_{i=1}^rG_{\alpha_i}, w}(\prod_{i=1}^r\mathbb{E}_{\alpha_i})
$$
induce the functors (we use the same notations for convenience, this should not cause any confusion by contexts.)
$$
\mathbf{m}^r_{\alpha_1,\cdots,\alpha_r}:\mathcal{Q}_{\alpha_1}\boxtimes\cdots\boxtimes\mathcal{Q}_{\alpha_r}\rightarrow \mathcal{Q}_{\alpha_1+\cdots+\alpha_r}
$$
and
$$
\mathbf{\Delta}^r_{\alpha_1,\cdots,\alpha_r}:\mathcal{Q}_{\alpha_1+\cdots+\alpha_r}\rightarrow \mathcal{Q}_{\alpha_1}\boxtimes\cdots\boxtimes\mathcal{Q}_{\alpha_r}
$$
with the commutative diagrams
$$\xymatrix{\mathcal{Q}_{\alpha_1}\boxtimes\cdots\boxtimes\mathcal{Q}_{\alpha_r}\ar[r]\ar[d]& \mathcal{Q}_{\alpha_1+\cdots+\alpha_r}\ar[d]\\
\mathcal{R}_{\alpha_1,\cdots,\alpha_r}\ar[r]& \mathcal{R}_{\alpha_1+\cdots+\alpha_r}}
$$
and
$$
\xymatrix{\mathcal{Q}_{\alpha_1+\cdots+\alpha_r}\ar[r]\ar[d]& \mathcal{Q}_{\alpha_1}\boxtimes\cdots\boxtimes\mathcal{Q}_{\alpha_r}\ar[d]\\
\mathcal{R}_{\alpha_1+\cdots+\alpha_r}\ar[r]&\mathcal{R}_{\alpha_1,\cdots,\alpha_r} .}
$$
Now the functor $\mathbf{S}:\mathcal{D}^b_{G_{\alpha},w}(\mathbb{E}_{\alpha})\rightarrow \mathcal{D}^b_{G_{\alpha},w}(\mathbb{E}_{\alpha})$ in Section 4 induces a functor $\mathbf{S}:\mathcal{Q}_{\alpha}\rightarrow \mathcal{Q}_{\alpha}$ with the commutative diagram
$$
\xymatrix{\mathcal{Q}_{\alpha}\ar[d]\ar[r]^{\mathbf{S}}&\mathcal{Q}_{\alpha}\ar[d]\\
\mathcal{R}_{\alpha}\ar[r]^{\mathbf{S}}&\mathcal{R}_{\alpha}.}
$$
As in Proposition \ref{antipode2}, we obtain the antipode over $K(\mathcal{Q}).$
\begin{Prop}
There exists a map $S:K(\mathcal{Q})\rightarrow K(\mathcal{Q})$ such that
$$\mathfrak{m}^t(S\otimes 1)\Delta^t([\mathcal{L}])=\mathfrak{m}^t(1\otimes S)\Delta^t([\mathcal{L}])=\left\{
\begin{array}{c}
\textrm{$0$ \,\,\,\,\,\,if $[\mathcal{L}]\neq[\mathbf{1}_0]$},\\
\textrm{$[\mathbf{1}_0]$ \,\,\,\,\,\,if $[\mathcal{L}]=[\mathbf{1}_0]$}.
\end{array}\right.$$

\end{Prop}

\section{The trace map and $F$-invariant representations}
In this section, we study the image of the trace map in Section 4. Let $\mathbb{K}=\overline{\mathbb{F}}_q$ and $F$ be the Frobenius automorphism. Let $Q$ be a quiver and $X=\mathbb{E}_{\alpha }$ be the variety of $\mathbb{K}Q$-modules of dimension vector $\alpha .$  We recall some definitions and notations in \cite{Hua2}. The map $F$ induces the morphism $F_{X}: X\rightarrow X$ by sending $x$ to $F_X(x)$, i.e., sending representations $M(x)$ to $M(F_X(x))$. Sometime we write $M(x)^{[q]}$ for $M(F_X(x))$. There exists the smallest positive integer $r$ such that $M(x)\cong M(x)^{[q^r]}$ over $\mathbb{K}=\overline{\mathbb{F}}_q.$ We call $\mathbb{F}_{q^r}$ the minimal field of definition of $M(x).$

Let $(\mathcal{F}, j_{\mathcal{F}})$ be a Weil complex in $\md^b_{G,w}(X)$ where $G=G_{\alpha }.$ If $x\in X(\mathbb{F}_q)$, then by definition, we have
$$\xymatrix{\mathcal{F}_x\ar[rr]^(0.3){(j_{\mathcal{F}}^{-1})_x}&&(F_X^*(\mathcal{F}))_x=\mathcal{F}_{F_X(x)}=\mathcal{F}_x.}$$
In general, if $x\in X(\mathbb{F}_{q^s})$ for $s\in \mathbb{N}$, we have
$$\xymatrix{\mathcal{F}_x\ar[r]^(0.4){(j_{\mathcal{F}}^{-1})_x}&\mathcal{F}_{F_X(x)}\ar[rr]^(0.6){(j_{F^*_X(\mathcal{F})}^{-1})_{F_X(x)}}&&\mathcal{F}_{F^2_X(x)}\ar[r]&\cdots\ar[r]&\mathcal{F}_{F^s_X(x)}=\mathcal{F}_x.}$$
We denote by $\phi_x$ the composition of these maps. In the same way, we have
$$
\xymatrix{(F_X^*(\mathcal{F}))_x\ar[rrr]^{(j_{\mathcal{F}}^{-1})_x\circ\phi_x\circ (j_{\mathcal{F}})_x}&&&(F_X^*(\mathcal{F}))_x}.
$$
Then we obtain the following commutative diagram
$$
\xymatrix{\mathcal{F}_x\ar[d]_{(j_{\mathcal{F}}^{-1})_x}\ar[rrr]^{\phi_x}&&&\mathcal{F}_x\ar[d]^{(j_{\mathcal{F}}^{-1})_x}\\
\mathcal{F}_{F_X(x)}\ar[rrr]^{(j_{\mathcal{F}}^{-1})_x\circ\phi_x\circ (j_{\mathcal{F}})_x}&&&\mathcal{F}_{F_X(x)}}
$$
and then
$$
\xymatrix{\mathcal{H}^i(\mathcal{F})_{\mid x}\ar[rr]^{F^s_{i,x}}\ar[d]_{\rho_x}&&\mathcal{H}^i(\mathcal{F})_{\mid x}\ar[d]^{\rho_x}\\
\mathcal{H}^i(\mathcal{F})_{\mid F_X(x)}\ar[rr]^{F^s_{i,F_X(x)}}&&\mathcal{H}^i(\mathcal{F})_{\mid F_X(x)},}
$$
where $\rho_x$ is the isomorphism of the stalk at $x$ of the $i$-th cohomology sheaves induced by $j_{\mathcal{F}}^{-1}$.
As a result, we obtain the characterization of the image of the trace map.
\begin{theorem}
For $s\geq 1$, the image of $\chi^{F^s}: \mathbf{K}_{w}\rightarrow \mathcal{CF}^{F^s}(Q)$ is the subspace spanned by the the functions $f\in \mathcal{CF}^{F^s}(Q)$ satisfying $f(x)=f(F_X(x))$ for any dimension vector $\alpha$, $X=\mathbb{E}_{\alpha}$ and $x\in \mathbb{E}^{F^s}_{\alpha}$.
\end{theorem}
\begin{proof}
The above diagram shows that $\mathcal{F}^s_{i, F_X(x)}=\rho_x\circ\mathcal{F}^s_{i,x}\circ\rho_x^{-1}$ and then $$tr(F^s_{i,x}, \mathcal{H}^i(\mathcal{F})_{\mid x})=tr(F^s_{i,F_X(x)}, \mathcal{H}^i(\mathcal{F})_{\mid F_X(x)}).$$
This gives the identity $\chi^{F^s}_{\mathcal{F}}(x)=\chi^{F^s}_{\mathcal{F}}(F_X(x)).$

Conversely, for $s\geq 1$, let $\mathcal{FCF}^{F^s}(Q)$ be the $\overline{\mathbb{Q}}_l$-vector subspace of $\mathcal{CF}^{F^s}(Q)$ generated by the functions $f\in \mathcal{CF}^{F^s}(Q)$ with $f(x)=f(F_X(x))$ for any dimension vector $\alpha$ and $x\in \mathbb{E}^{F^s}_{\alpha}$. Then $\mathcal{FCF}^{F^s}(Q)$ is generated by  characteristic functions $1_{\mo_{F,M}}$ of  $G_{\alpha}$-$F$-orbits $\mo_{F,M}$ for some dimension vector $\alpha$ and $M\in \mathbb{E}^{F^s}_{\alpha}$.  Let $j: \mo_{F,M}\rightarrow \mathbb{E}_{\alpha}$ be the natural embedding. Then the complex $\mathcal{C}_M=j_{!}(\overline{\mathbb{Q}}_l[\mathrm{dim}\mo_{F, M}])$ satisfies that $\chi^{F^t}([\mathcal{C}_M])=1_{\mo^{F^t}_{F, M}}$ for $t\in \mathbb{N}$. In particular, for $t<s$, $\mo^{F^t}_{F, M}=\emptyset$ and $\chi^{F^t}([\mathcal{C}_M])=0.$

Given $b\in \overline{\mathbb{Q}}_l$, then there exists a generalized Weil complex $\overline{\mathbb{Q}}_l^{(b)}$ over $Spec(\mathbb{F}_q)$ such that the Frobenius morphism acts as the multiplication by $b$: $\overline{\mathbb{Q}}_l\rightarrow \overline{\mathbb{Q}}_l$ (see \cite[Chapter 1]{KW}). Consider the natural projection $\pi: X\rightarrow Spec(\mathbb{F}_q)$. The complex $\mathcal{L}_s(b)=\pi^{*}(\overline{\mathbb{Q}}_l^{(\sqrt[s]{b})})$ satisfies that $\chi^{F^s}([\pi^{*}(\overline{\mathbb{Q}}_l^{(\sqrt[s]{b})})])$ is the constant function over $\mathbb{E}^{F^s}_{\alpha}$ with the value $b.$ Then, we obtain
$$
\chi^{F^s}(\mathcal{C}_M\otimes \mathcal{L}_s(b))=b\cdot 1_{\mo_{F,M}}.
$$
Hence, $\chi^{F^s}: \mathbf{K}_{w}\rightarrow \mathcal{FCF}^{F^s}(Q)$ is surjective.

We complete the proof.
\end{proof}

The map $\chi^{F^s}$ can be restricted to $\mathbf{I}_{w}$. Define $\chi^{F^s}_{\mathbf{I}_w}: {_{\mathbb{Q}(v)}\mathbf{I}}_{w}\rightarrow \mathcal{CF}^{F^s}(Q)$ by sending $[M]\otimes f(v)$ to $f((-\sqrt{q})^s)\cdot \chi^{F^s}([M]).$ Set $\chi^{F^s}_{\mathbf{I}_w, \alpha}$ to be the restriction of $\chi^{F^s}_{\mathbf{I}_w}$ to ${_{\mathbb{Q}(v)}K}_{im,G_{\alpha},w}(\mathbb{E}_{\alpha})$ for any dimension vector $\alpha$.
We denote by $$\mathbb{M}^F(\alpha, q^s)=\mathrm{Im} \chi^{F^s}_{\mathbf{I}_w, \alpha}\otimes_{\mathbb{Q}((-\sqrt{q})^s)}\overline{\mathbb{Q}}_l.$$ As the natural corollary of the above theorem, we obtain the following result.
\begin{theorem}
$\mathbb{M}^F(\alpha, q^s)=\mathcal{FCF}^{F^s}(Q).$
\end{theorem}

Analogous to the notations $M_{Q}(\alpha , q^s), I_{Q}(\alpha , q^s)$ and $A_{Q}(\alpha , q^s)$ in \cite{Hua}, we define the $F_X$-versions as follows:
\begin{eqnarray}
  M^F_{Q}(\alpha , q^s) &=& \mbox{the number of }G_{\alpha}\mbox{-}F_X\mbox{-orbits of  representations  }\nonumber \\
    &&\mbox{of }Q  \mbox{ over } \mathbb{F}_{q^s} \mbox{ with dimension vector } \alpha,     \nonumber   \\
 I^F_{Q}(\alpha , q^s) &=& \mbox{the number of }G_{\alpha}\mbox{-}F_X\mbox{-orbits of indecomposable  }\nonumber \\
    &&\mbox{representations of }Q  \mbox{ over } \mathbb{F}_{q^s} \mbox{ with dimension vector } \alpha,     \nonumber   \\
 A^F_{Q}(\alpha , q^s) &=& \mbox{the number of }G_{\alpha}\mbox{-}F_X\mbox{-orbits of absolutely indecomposable}\nonumber \\
    &&\mbox{representations of }Q  \mbox{ over } \mathbb{F}_{q^s} \mbox{ with dimension vector } \alpha.     \nonumber
\end{eqnarray}

The following proposition is clear.
\begin{Prop}
With the above notations, we have
$$
\mathrm{dim}_{\overline{\mathbb{Q}}_l}\mathbb{M}^F(\alpha, q^s)=M^F_{Q}(\alpha , q^s).
$$
\end{Prop}
Let $M^{min}_{Q}(\alpha , q^r)$ be the number of isomorphism classes of representations of $Q$ over $\mathbb{K}=\overline{\mathbb{F}}_q$ with dimension vector $\alpha $ and minimal field of definition $\mathbb{F}_{q^r}.$
By definition, we have
$$
M_Q(\alpha , q^s)=\sum_{r\mid s}M^{min}_Q(\alpha , q^r), \quad \quad M^F_Q(\alpha , q^s)=\sum_{r\mid s}\frac{1}{r}M^{min}_Q(\alpha , q^r).
$$
The M\"{o}bius inversion of the first identity is
$$
M^{min}_Q(\alpha , q^s)=\sum_{r\mid s}\mu(\frac{s}{r}) M_Q(\alpha , q^r).
$$
Then we have
$$
\quad \quad M^F_Q(\alpha , q^s)=\sum_{r\mid s}\frac{1}{r}\sum_{t\mid r}\mu(\frac{r}{t}) M_Q(\alpha , q^t).
$$

In \cite{Hua2}, the author shows that $M_{Q}(\alpha , q)\in \mathbb{Q}[q].$ The following proposition is a direct corollary.
\begin{Prop}
$M^F_{Q}(\alpha , q^s)\in \mathbb{Q}[q].$
\end{Prop}
For example, $$M^F_{Q}(\alpha , q)=M_{Q}(\alpha , q), \quad M^F_{Q}(\alpha , q^2)=\frac{1}{2}M_{Q}(\alpha , q^2)+\frac{1}{2}M_{Q}(\alpha , q)$$ and
$$
M^F_{Q}(\alpha , q^3)=\frac{1}{3}M_{Q}(\alpha , q^3)+\frac{1}{6}M_{Q}(\alpha , q^2)+\frac{1}{2}M_{Q}(\alpha , q).
$$
\begin{Prop}
There exists a polynomial $f(t)\in \mathbb{Q}[t]$ such that for any two prime numbers $q, q'$ and $s\in \mathbb{N}$, $f(q^s)=M^F_{Q}(\alpha , q^s)$  and $f(q^s)=M^{F'}_{Q}(\alpha , {q'}^s)$
where $F'$ is the Frobenius automorphism of $\mathbb{K}'=\overline{\mathbb{F}}_{q'}$, i.e., $F'(x)=x^{q'}$.
\end{Prop}
The proposition follows from that $M_{Q}(\alpha , q)$ is a polynomial in $q$ with rational coefficients (\cite[Section 4]{Hua2}).

\bigskip
\par\noindent {\bf Acknowledgments.}
The authors are very grateful to Hiraku Nakajima  for telling us that Lusztig's restriction is a hyperbolic localization and Reference \cite{Braden}. The second named author thanks Sheng-Hao Sun for explaining the contents in Reference \cite{KW} and many helpful
comments.

\end{document}